\documentclass[a4paper,10pt,reqno, english]{amsart} 

\usepackage[utf8]{inputenc}
\usepackage{url,paralist}
\usepackage{caption}
\usepackage{subcaption}
\usepackage{makecell} 
\usepackage[colorlinks=true,urlcolor=blue,citecolor=magenta]{hyperref}
\usepackage{color}
\usepackage{float}
\usepackage{dsfont}
\usepackage{tikz}
\usepackage{amsfonts,amssymb,enumerate}
\usepackage{graphicx}
\usepackage{amsmath,bm,amsthm}
\usepackage{mathtools}
\usetikzlibrary{arrows.meta}
\usepackage[arrow,curve,matrix,tips,2cell]{xy}  
  \SelectTips{eu}{10} \UseTips
  \UseAllTwocells

\theoremstyle{plain}
\newtheorem{theorem}{Theorem}[section]
\newtheorem{lemma}[theorem]{Lemma}

\newtheorem{corollary}[theorem]{Corollary}

\newtheorem*{theorem*}{Theorem}
\newtheorem*{mproblem}{Main Problem}

\newtheorem*{ghrproblem}{The Gr\"unbaum--Hadwiger--Ramos mass partition problem}
\newtheorem*{r-conjecture}{The Ramos conjecture}

\newtheorem*{claim*}{Claim}

\theoremstyle{definition}

\newcommand{\R}{\mathbb{R}}

\newcommand{\Z}{\mathbb{Z}}
\newcommand{\F}{\mathbb{F}}
\newcommand{\B}{\mathrm{B}}
\newcommand{\E}{\mathrm{E}}
\newcommand{\pt}{\mathrm{pt}}
\newcommand{\CC}{\mathcal{C}}
\newcommand{\M}{\mathcal{M}}
\newcommand{\HH}{\mathcal{H}}
\newcommand{\II}{\mathcal{I}}
\newcommand\Sym{\mathfrak{S}}
\newcommand\Sympm{\mathfrak{S}^{\pm}}

\newcommand{\im}{\operatorname{im}}

\newcommand{\id}{\operatorname{id}}
\newcommand{\ind}{\operatorname{Index}}

\newcommand{\Gr}[2]{ G_{#1}( \mathbb{R}^{#2}) }

\newcommand{\Index}[4]{\ind_{#3}^{#2}(#1;#4) }

\newcommand*\bigC[1]{ \big(  #1 \big) }

\title[Topology of the GHR-problem for mass assignments]{Topology of the Gr\"unbaum--Hadwiger--Ramos problem for mass assignments}

\author[Blagojevi\'c]{Pavle V. M. Blagojevi\'{c}} 
\thanks{The research by Pavle V. M. Blagojevi\'{c} leading to these results has
        received funding from the Serbian Ministry of Education and Science.}
\address{Inst. Math., FU Berlin, Arnimallee 2, 14195 Berlin, Germany\hfill\break
\mbox{\hspace{4mm}}Mat. Institut SANU, Knez Mihailova 36, 11001 Beograd, Serbia}
\email{blagojevic@math.fu-berlin.de} 
\author[Calles]{Jaime Calles Loperena}
\thanks{The research by Jaime Calles Loperena leading to these results has received funding from CONACYT doctoral scholarship, CONACYT project grant CB 217392, and PAPIIT grants IA100119 and IN100221.}
\address{Centro de Ciencias Matem\'aticas, UNAM Campus Morelia, Morelia, Mexico}
\email{jcalles@matmor.unam.mx}
\author[Crabb]{Michael C. Crabb} 
\address{Institute of Mathematics, University of Aberdeen, Aberdeen AB24 3UE, UK}
\email{m.crabb@abdn.ac.uk}
\author[Dimitrijevi\'{c} Blagojevi\'{c}]{Aleksandra S. Dimitrijevi\'{c} Blagojevi\'{c}} 
\thanks{}
\address{Beograd, Serbia}
\email{aleksandra1973@gmail.com}

\begin{document}

\date{}

\maketitle

\begin{abstract}

In this paper, motivated by recent work of Schnider and  Axelrod-Freed \& Sober\'on, we study an extension of the classical Gr\"unbaum--Hadwiger--Ramos mass partition problem to the  mass assignments.
Using the Fadell--Husseini index theory we prove that for a given family of $j$ mass assignments $\mu_1,\dots,\mu_j$ on the Grassmann manifold $G_{\ell}(\R^d)$ and the given integer $k\geq 1$ there exist a linear subspace $L\in G_{\ell}(\R^d)$ and $k$ affine hyperplanes in $L$ equiparts the masses $\mu_1^L,\dots,\mu_j^L$ assigned to the subspace $L$, provided that $d\geq j + (2^{k-1}-1)2^{\lfloor\log_2j\rfloor}$. 

\end{abstract}

\section{Introduction and statement of main results}
\label{sec : Introduction and statement of main results}

In this paper, motivated by the recent work of Patrick Schnider \cite{Schnider-2019} and Ilani Axelrod-Freed \& Pablo Sober\'on \cite{Soberon},  we consider an extension of the classical Gr\"unba\-um--Hadwiger--Ramos mass partition problem to the mass assignments.
First we set up the terminology to be used in the paper.

\medskip
A {\em mass} is assumed to be a finite Borel measure on a Euclidean space which vanishes on each affine hyperplane.  
For example, masses in $\R^d$ are measures given by the $d$-dimensional volume of a proper convex body, induced by lengths of intervals on a moment curve in $\R^d$, and measures given by a finite collection of pairwise disjoint  balls.

\medskip
Let $X$ be a locally compact Hausdorff space, and let $M_+(X)$ denote the set of all finite Borel  measures on $X$.
For a definition of the Borel measure on a topological space consult for example \cite[Def.\, 2.15]{Rudin-1987}.
The {\em weak topology} on $M_+(X)$ is defined to be the minimal topology such that for every bounded and upper semi-continuous function $f \colon X \longrightarrow \R$, the induced function $M_+(X)\longrightarrow\R$, $\nu \longmapsto \int f d\nu$, is upper semi-continuous.
Here minimality is considered with respect to the inclusion of families of (open) subsets of $X$. 
In the case when $X=\R^{\ell}$ we denote by $M_+'(\R^{\ell})\subseteq M_+(\R^{\ell})$ the subspace of all masses on $\R^{\ell}$.
For more details about spaces of measures and related notions consult \cite{Topsoe1970}.

\medskip
Let $\Gr{\ell}{d}$, $0\leq\ell\leq d$, denote the Grassmann manifold of all $\ell$-dimensional linear subspaces of $\R^d$.
Consider the following fibre bundle
\begin{equation}
\label{eq: mass_assignment}	
\xymatrix{
M_+'(\R^{\ell})\ar[r] & \mathcal{M}_+'(\ell,d)\ar[r]^-{\rho} & \Gr{\ell}{d}
}
\end{equation}
where the total space is given by
\[
\mathcal{M}_+'(\ell,d):= \{ (L, \nu) \mid L \in \Gr{\ell}{d} \; \text{and} \; \nu \in M_+'(L) \}
\]
and the map $\rho$ by $( L, \nu )\longmapsto L$.
A {\em mass assignment} $\mu$
on $\Gr{\ell}{d}$ is a cross-section of the  fibre bundle \eqref{eq: mass_assignment}, which assigns to each subspace $L \in \Gr{\ell}{d}$ a mass $\mu^{L}$ on $L$. 
Mass assignments on $\Gr{\ell}{d}$ are for example, projections of masses in $\R^d$ to $\ell$-dimensional linear subspaces and volumes of intersections of proper convex body in $\R^d$ with $\ell$-dimensional linear subspaces.

\medskip
We use a definition of the mass assignment on a Grassmann manifold as a cross section of the fibre bundle \eqref{eq: mass_assignment}, which, while being equivalent, differs from the treatment of this notion in the work of Schnider \cite[p.\,1193]{Schnider-2019} and Axelrod-Freed \& Sober\'on \cite[Def.\,1]{Soberon}.
Schnider defines a mass assignment on $\Gr{\ell}{d}$ as a continuous function $\Gr{\ell}{d}\longrightarrow M_{\ell}$ where $M_{\ell}$ is the space of all $\ell$-dimensional mass distributions equipped with the weak topology.
On the other hand Axelrod-Freed \& Sober\'on consider the affine Grassmannian $A_{\ell}(\R^d)$ of all $\ell$-dimensional affine subspaces of $\R^d$ --- the total space of the tautological vector bundle $\gamma_{d-\ell}^d$ over the Grassmann manifold $\Gr{d-\ell}{d}$.
In addition, for every $L\in A_{\ell}(\R^d)$ they equip the set $M_{\ell}(L)$, of all finite  measures on $V$ absolutely continuous with respect to the Lebesgue measure on $L$, with the weak topology.
Then an $\ell$-dimensional mass assignment is a cross section of the fibre bundle
\[
\xymatrix{
M_{\ell}(L)\ar[r] & \{ (L,\nu) : L\in A_{\ell}(\R^d)  \; \text{and} \; \nu \in M_{\ell}(L) \}\ar[r] & A_{\ell}(\R^{d}).
}
\]
All the definitions, after appropriate adjustments, are equivalent and more importantly induce identical topological questions.

\medskip
Let $v\in S^{d-1}$ be a unit vector in $\R^d$, and let $a\in\R$.
The {\em oriented affine hyperplane} $H(v;a)$ in $\R^d$, oriented by $v$ on the distance $a$ from the origin (in direction $v$), determines the associated affine hyperplane 
\[
H_{v;a}:=\{x\in\R^d \colon \langle x,v\rangle =a\}, 
\] 
and, in addition, two closed half-spaces which are denoted by  
\[
H_{v;a}^{0}:=\{x\in \R^d \colon\langle x,v\rangle \geq a\}\qquad\text{and}\qquad
H_{v;a}^{1}:=\{x\in\R^d \colon \langle x,v\rangle \leq a\}.
\]
In other words, the oriented affine hyperplane $H(v;a)$ can be identified with the triple $(H_{v;a},H_{v;a}^{0},H_{v;a}^{1})$.
The following equalities on the level of sets hold: $H_{v;a}=H_{-v;-a}$, $H_{v;a}^{0}=H_{-v;-a}^{1}$ and $H_{v;a}^{1}=H_{-v;-a}^{0}$.

\medskip
A \emph{$k$-arrangement} $\HH$ in $\R^d$ is an ordered collection of $k$ oriented affine hyperplanes $\HH=\big( H(v_1;a_1), \ldots, H(v_k;a_k) \big)$.
The \emph{orthant} determined by the $k$-arrangement $\HH$ and the element $\alpha=(\alpha_1,\ldots,\alpha_k)\in \Z_2^k=\{0,1\}^k$ of the abelian group $\Z_2^k$ is the following intersection of closed halfspaces
\[
\mathcal{O}_{\alpha}^{\HH}=H_{v_1;a_1}^{\alpha_{1}}\cap\cdots\cap H_{v_k;a_k}^{\alpha_{k}}.
\]

\medskip
A $k$-arrangement $\HH$ \emph{equiparts} a collection of masses $\M=(\mu_1,\ldots,\mu_j)$ if for every element $\alpha \in \Z_2^k$ and every $r\in\{1,\ldots,j\}$ we have that:
\[
\mu_{r} (\mathcal{O}_{\alpha}^{\HH})=\tfrac{1}{2^{k}}\mu_{r}(\R^d).
\]
This can be achieved only in the case when $k\le d$.
For that reason it is always silently assumed that the number of hyperplanes we consider does not exceed the dimension of the ambient space.
 
\subsection{The Gr\"unbaum--Hadwiger--Ramos problem for masses}
\label{subsec : The Grunbaum--Hadwiger--Ramos problem for masses}
The study of mass partition problems by affine hyperplanes started with a classical result, the so called ham sandwich theorem, conjectured by Hugo Steinhaus \cite[Problem 123]{Mauldin-1981}, and proved by Karol Borsuk in 1938; for details about the history see \cite{Beyer-Zardecki-2004}.
The ham sandwich theorem states that for any collection of $d$ masses living in a $d$-dimensional Euclidean space there exists an affine hyperplane which equiparts the collection that cuts each of the masses into two equal parts.

\medskip
A few decades later Branko Gr\"unbaum in his paper \cite{Grunbaum-1960} asked the following question: Is it possible to equipart a single mass in $\R^d$ by a $d$-arrangement?
He noted that, while the answer in the case of a line is obviously positive, the positive answer for the case of the plane follows directly from the ham sandwich theorem.
The positive answer to the Gr\"unbaum's question in the case $d=3$ was given by Hugo Hadwiger \cite{Hadwiger-1966} in 1966 as a consequence of his result: For any collection of two masses in $\R^3$ there exists a $2$-arrangement which equiparts the collection. 
In 1984 David Avis \cite{Avis-1984} showed that in every dimension $d\geq 5$ there is a mass which cannot be equiparted by a $d$-arrangement.
The case of dimension $4$, to this very day, is still open, meaning that we would like to know: Is it possible to equipart one mass in $\R^4$ with a $4$-arrangement?

\medskip
In 1996 Edgar Ramos \cite{Ramos-1996} proposed the following extension of the  Gr\"unbaum hyperplane mass partition problem.
  
\begin{ghrproblem}
\label{GruenbaumProblem}
Determine the minimal dimension $d=\Delta (j,k)$ of a Euclidean space $\R^d$ such that for every collection of $j$ masses in $\R^d$ there exists a $k$-arrangement equiparting the collection of masses.
\end{ghrproblem}

In particular, the ham sandwich theorem is equivalent to the equality $\Delta(d,1)=d$, while the results of  Gr\"unbaum and Hadwiger imply that $\Delta(1,2)=2$, $\Delta(2,2)=3$ and $\Delta(1,3)=3$.
Based on the ideas of Avis, Ramos derived the following lower bound for the function  $\Delta (j,k)$:\[
\tfrac{2^k-1}{k}j\leq \Delta(j,k).
\]
The lower bound transformed into the following conjecture.

\begin{r-conjecture}  
$\Delta(j,k)=\lceil \tfrac{2^k-1}{k}j\rceil$ for every $j\geq 1$ and $k\geq 1$.
\end{r-conjecture}

An upper bound for the function $\Delta(j,k)$ was obtained in 2006 by Peter Mani-Levitska, Sini\v{s}a Vre\'{c}ica \& Rade \v{Z}ivaljevi\'c in \cite[Thm.\,39]{ManiLevitska-Vrecica-Zivaljevic-2006}:
\[
      \Delta(j,k)\ \le\ j + (2^{k-1}-1)2^{\lfloor\log_2j\rfloor}.
\]
The only instance in which lower and upper bounds coincide is in the case when $k=2$ and $j=2^{t+1}-1\geq 1$.

\medskip
Over the years, using a variety of methods from equivariant algebraic topology, different groups of authors studied the conjecture of Ramos.
Despite considerable effort the conjecture has been confirmed rigorously only in a few special cases; for more details on the history and a discussion of solution methods consult a critical review \cite{Blagojevic-Frick-Haase-Ziegler-2018}, and for the currently best known results see \cite{Blagojevic-Frick-Haase-Ziegler-2016}.

\subsection{The Gr\"unbaum--Hadwiger--Ramos problem for mass assignments}
\label{subsec : The Grunbaum--Hadwiger--Ramos problem for mass assignments}

The problem we consider in this paper is the following extension of the  Gr\"unbaum--Hadwiger--Ramos problem to mass assignments.

\medskip
A $4$-tuple of natural numbers $(d,\ell,k,j)$ where $1\leq\ell\leq d$ is called {\em mass  assignment admissible} if for every collection of $j$ mass assignments $\M=( \mu_1, \ldots , \mu_j )$ on the Grassmann manifold $\Gr{\ell}{d}$ there exist a vector subspace $L \in  \Gr{\ell}{d}$ and a $k$-arrangement $\HH^L$ in $L$ which equiparts the collection of $j$ masses $(\mu_1^L, \ldots , \mu_j^L)$.
Observe that $(d,\ell,k,j)$ is mass assignment admissible only when $k\le \ell$.
The case $\ell=d$ coincides with the classical  Gr\"unbaum--Hadwiger--Ramos problem.
  
\begin{mproblem}
\label{GruenbaumProblemMassAssignments}
Determine all mass assignment admissible $4$-tuples.
\end{mproblem}

The reason for considering an extension of the classical  Gr\"unbaum--Hadwiger--Ramos problem lies in the fact that now instead of one collection of masses we have a family of collections of masses parametrised by a topological space.
Hence it reasonable to expect that the topology of the parameter space will come into play essentially allowing us to equipart more masses than we could in the classical case.

\medskip
Schnider, in his recent publication \cite[Thm.\,2]{Schnider-2019}, showed that any $4$-tuple of the form $( d, \ell , 1 , d )$ where $ 1 \leq \ell \leq d $ is  mass  assignment admissible.
In other words, he showed that for every collection of $d$ mass assignment $(\mu_1,\dots,\mu_d)$ on $\Gr{\ell}{d}$ there exists an $\ell$-dimensional linear subspace $L\subseteq\R^d$ and an affine hyperplane in $L$ which equiparts all $d\geq\ell\geq 1$  masses $(\mu_1^L,\dots,\mu_d^L)$.
Thus, in the case of the Gr\"unbaum--Hadwiger--Ramos problem for mass assignments we can expect to partition more masses compared to the classical case.

\medskip
In this paper, we prove the following general algebraic criterion for the assignment admissibility which also yields multiple corollaries.
For the statement of the theorem we introduce the following truncated polynomial ring
\begin{equation}
\label{eq: definition of R}
R_{d,\ell,k}:=\F_2[x_1,\dots,x_k, w_1, \dots ,w_{\ell}, \overline{w}_1, \dots , \overline{w}_{d-\ell}]/I_{d,\ell}	
\end{equation}
where $\deg(x_1)=\dots=\deg(x_k)=1$, $\deg( w_s) = s$, $\deg( \overline{w}_r) = r$ for $1\leq s\leq \ell$, $1\leq r\leq d-\ell$, and $I_{d,\ell}$ is the ideal generated by the following $d$ polynomials
\[
\sum_{s=\max\{0,r+\ell-d\}}^{\min\{r,\ell\}}w_s\cdot \overline{w}_{r-s}, \qquad 1\leq r\leq d,
\]
in variables $w_1, \dots ,w_{\ell}, \overline{w}_1, \dots , \overline{w}_{d-\ell}$.
Note that the ring $R_{d,\ell,k}$ is isomorphic with the cohomology ring $H^*(\B(\Z_2^k)\times \Gr{\ell}{d};\F_2)$, and it can also be seen as a polynomial ring over the ring $H^*(\Gr{\ell}{d};\F_2)$, that is
\[
R_{d,\ell,k}\cong (\F_2[w_1, \dots ,w_{\ell}, \overline{w}_1, \dots , \overline{w}_{d-\ell}]/I_{d,\ell})[x_1,\dots,x_k].
\]
Indeed, the K\"unneth formula for field coefficients \cite[Thm.\,VI.3.2]{Bredon-1997}, in combination with the description of the cohomology of $\B(\Z_2^k)$ in \cite[Thm.\,II.4.4]{Adem-Milgram-2004}, and the cohomology of $\Gr{\ell}{d}$  in \cite[p. 190]{Borel-1953}, imply that
\begin{align*}
H^*(\B(\Z_2^k)\times \Gr{\ell}{d};\F_2)&\cong H^*(\B(\Z_2^k);\F_2)\otimes H^*(\Gr{\ell}{d};\F_2)\\
&\cong	\F_2[x_1,\dots,x_k]\otimes \F_2[w_1, \dots ,w_{\ell}, \overline{w}_1, \dots , \overline{w}_{d-\ell}]/I_{d,\ell}.
\end{align*}

\medskip
The main result of the paper is the following theorem which we prove in Section \ref{Proof of Main Theorem} based on the configuration test map scheme developed in Section \ref{sec : From the Grunbaum--Hadwiger--Ramos problem for mass assignments to a parametrized Boursuk--Ulam type question} and the computations done in Section \ref{sec : The Fadell--Husseini index of the configuration and test space}.

\begin{theorem}
\label{theorem : Main}	
Let $d\geq 2$, $k\geq 1$, $j\geq 1$ and $\ell\geq 1$ be integers.
A $4$-tuple of natural numbers $(d,\ell,k,j)$ where $1\leq\ell\leq d-1$ is mass assignment admissible if the element  
\begin{align*}
e_{k,j} &:=
\prod_{i=1}^{k} x_i^{j-1} \cdot \prod_{(\alpha_{1},\dots, \alpha_{k})\in\F_2^k{\setminus}\{ (0,\ldots,0), (1,\ldots,0), \ldots, (0,\ldots,1) \}} (\alpha_{1}x_1 + \cdots + \alpha_{k}x_k)^{j} \\
& =\frac{1}{x_1\cdots x_k}\prod_{(\alpha_{1},\dots, \alpha_{k})\in\F_2^k{\setminus}\{ (0,\ldots,0) \}} (\alpha_{1}x_1 + \cdots + \alpha_{k}x_k)^{j}
\end{align*}
of the ring $R_{d,\ell,k}$ is not contained in the ideal
\[
\II_{d,\ell,k}:=\Big\langle \sum_{s = 0}^{\ell}  x_{r}^{s} \, w_{\ell-s} 	\: : \:	1 \leq r \leq k \Big\rangle.
\]
\end{theorem}

\medskip
As the first consequence of Theorem \ref{theorem : Main} we recover the ham sandwich type result of Schnider  \cite[Thm.\,2]{Schnider-2019}.
For the proof see Section \ref{Proof of Corollary ham-sandwich}. 

\begin{corollary}
\label{cor : ham-sandwich}
Let $d\geq 2$ be an integer. 
Every $4$-tuple of the form $( d, \ell , 1 , d )$, where $ 1 \leq \ell \leq d $, is mass assignment admissible.
\end{corollary}

\medskip
In order to state a consequence of Corollary \ref{cor : ham-sandwich} we introduce a special type of a mass assignment.
Let $s\colon \Gr{\ell}{d}\longrightarrow E(\gamma_{\ell}^d)$ be a section of the tautological vector bundle $\gamma_{\ell}^d$ over $\Gr{\ell}{d}$.
For a positive real number $\varepsilon>0$ the section $s$ defines a mass assignment $\mu_s$ given by $L\longmapsto B_L(s(L),\varepsilon)$.
Here $B_L(s(L),\varepsilon)$ denotes the Euclidean closed ball in $L$ with center at $s(L)$ and radius $\varepsilon$, or in other words the mass induced by this ball. 
Since any closed Euclidean ball is cut into halves of equal volume by an affine hyperplane if and only if this hyperplane passes through the centre of the ball, we get the following statement as a direct consequence of   Corollary \ref{cor : ham-sandwich}.

\begin{corollary}
\label{cor : ham-sandwich-1}
Let $d\geq 2$ and $1\leq \ell\leq d-1$ be integers.
For every collection of $d$ sections  $s_1,\dots,s_d\colon \Gr{\ell}{d}\longrightarrow E(\gamma_{\ell}^d)$ of the tautological vector bundle $\gamma_{\ell}^d$ over $\Gr{\ell}{d}$, there exists a subspace $L\in \Gr{\ell}{d}$ and an affine hyperplane $H$ in $L$ such that $s_1(L)\in H,\dots,s_d(L)\in H$.
\end{corollary}

\medskip
While Corollary \ref{cor : ham-sandwich-1} is an easy consequence of Corollary \ref{cor : ham-sandwich} and while it does not use much information about the Grassmann manifold $\Gr{\ell}{d}$, one can deduce more by using some additional information about the Stiefel--Whitney classes of $\gamma_{\ell}^d$.
More precisely, the so called intersection lemma \cite[Lem.\,4.3]{Blagojevi-Matschke-Ziegler-2011} in combination with the fact that $w_{\ell}(\gamma_{\ell}^d)^{d-\ell}\neq 0$ does not vanish, see \cite[Lem.\,1.2]{Hiller-1980}, yields the following fact: For every collection of $d-\ell$ sections $s_1,\dots,s_{d-\ell}$ of $\gamma_{\ell}^d$ there exists a subspace $L\in \Gr{\ell}{d}$ with the property that  $s_1(L)=\dots=s_{d-\ell}(L)$.
In particular, the points $s_1(L), \dots,s_{d-\ell}(L),s_{d-\ell+1}(L),\dots s_{d}(L)$ lie on a hyperplane in $L$.

\medskip 
The major consequence of Theorem  \ref{theorem : Main} is the following numerical criterion for a $4$-tuple $(d,\ell,k,j)$ to be mass  assignment admissible.
The proof of this result is given in Section \ref{Proof of Main numerical theorem}.

\begin{theorem}
\label{theorem : Main numerical}	
Let $d\geq 2$, $k\geq 1$, $j\geq 1$, $\ell\geq 1$ and $t\geq 0$, $r\geq 0$ be integers with $1\leq k\leq \ell\leq d$.
If $j=2^t+r$ with $0\leq r\leq 2^{t}-1$, and $d\geq 2^{t+k-1}+r$, then the $4$-tuple $(d,\ell,k,j)$ is mass assignment admissible.  
\end{theorem}

An interesting observation is that the condition for the $4$-tuple $(d,\ell,k,j)$, $1\leq k\leq \ell\leq d-1$, to be admissible given by Theorem \ref{theorem : Main numerical} does not depend on $\ell$ whatsoever.
Furthermore, using push-forward measures along the orthogonal projections onto $\ell$-dimensional subspaces, we can see that if the $4$-tuple $(d,\ell,k,j)$ is mass assignment admissible, then $d\geq \Delta(k,j)$ and so  $d\geq \lceil \tfrac{2^k-1}{k}j\rceil$.
Hence, a necessary condition for $(d,\ell,k,j)$ to be mass assignment admissible also does not involve $\ell$.
Consequently, as in the case of the classical Gr\"unbaum--Hadwiger--Ramos mass partition problem, we get that the upper and the lower bounds for dimension $d$ coincide for $k=2$ and $j=2^{t+1}-1$.
In other words we get the following corollary.

\begin{corollary}
	For every integer $t\geq 1$ we have that 
	\[
	\Delta(2,2^{t+1}-1) =\min\{ d: (d,\ell,2,2^{t+1}-1)\text{ \rm is mass assignment admissible}\}.
	\]
\end{corollary}
 
\medskip 
Finally, in Table \ref{table Comparison}  we compare the result of Theorem \ref{theorem : Main numerical} with the corresponding known results for the classical Gr\"unbaum--Hadwiger--Ramos mass partition problem for some concrete choices of parameters  $(d, \ell,j,k)$.
For that we recall the known equalities 
$
	\Delta(2^t +1, 2) = 3\cdot 2^{t-1} + 1
$
and
$
	\Delta(2^{t+1} -1, 2) = 3\cdot 2^t - 1
$
where $t \geq 2$. 
For more details of these two results see for example \cite{Blagojevic-Frick-Haase-Ziegler-2018}.

\begin{table}[h!]
	\centering 
	 \begin{tabular}{||m{4cm} | m{6cm} ||} 
		 \hline  
		  $\Delta( \textcolor{blue}{j},\textcolor{magenta}{k} ) $ &  Admissible $4$-tuples $(d, \ell, \textcolor{magenta}{k}, \textcolor{blue}{j}=2^t+r )$   \\ [1ex] 
		 \hline\hline
		 $\Delta( \textcolor{blue}{2} , \textcolor{magenta}{2} ) =  3 $	&	
		 	$(  8	,  3 ,  \textcolor{magenta}{2}  ,	\textcolor{blue}{  4  }	  )$ {\tiny( considering $t=2$ and $r=0$)}  \\
		 \hline
		 $\Delta( \textcolor{blue}{1} , \textcolor{magenta}{3} ) =  3 $	&	
		 	$(  9	,  3 ,  \textcolor{magenta}{3}  ,	\textcolor{blue}{  3  }	  )$ {\tiny( considering $t=1$ and $r=1$)}  \\
		 \hline
		 $\Delta( \textcolor{blue}{5} , \textcolor{magenta}{2} ) =  8 $	&	
		 	$(  11	,  8 ,  \textcolor{magenta}{2}  ,	\textcolor{blue}{  7  }	  )$ {\tiny( considering $t=2$ and $r=3$)}  \\
		 \hline
		 $\Delta( \textcolor{blue}{9} , \textcolor{magenta}{2} ) =  14 $ & 
		 	$(  17	,  14 ,  \textcolor{magenta}{2}  ,	\textcolor{blue}{  15  }	)$ {\tiny( considering $t=3$ and $r=7$)}  \\ 	 
		 \hline
		 $\Delta( \textcolor{blue}{7} , \textcolor{magenta}{2} ) =  11 $ & 
		 	$(  23	,  11 ,  \textcolor{magenta}{2}  ,	\textcolor{blue}{  15  }	)$ {\tiny( considering $t=3$ and $r=7$)} \\ 
		 \hline
		 $\Delta( \textcolor{blue}{15} , \textcolor{magenta}{2} ) =  23 $ & 
		 	$(  47	,  23 ,  \textcolor{magenta}{2}  ,	\textcolor{blue}{  31  }	  )$ {\tiny( considering $t=4$ and $r=15$)} \\ 		 
		 \hline
	\end{tabular}
	\caption{Comparison between the classical Gr\"unbaum--Hadwiger\allowbreak--Ramos mass partition problem and Theorem 
	\ref{theorem : Main numerical}. }
	\label{table Comparison}
\end{table}

The result in the Table \ref{table Comparison} illustrate the fact that the Gr\"unbaum--Hadwiger--Ramos problem for mass assignments can be solved positively for more masses than one can hope for in the classical case. 
This should not be a surprise since we are now given families of masses from which we can chose a particular collection to equipart.

\section{From the Gr\"unbaum--Hadwiger--Ramos problem for mass assignments to a parametrized Boursuk--Ulam type question}
\label{sec : From the Grunbaum--Hadwiger--Ramos problem for mass assignments to a parametrized Boursuk--Ulam type question}

In this section we demonstrate how the Gr\"{u}nbaum--Hadwiger--Ramos problem for mass assignments induces a parametrized Boursuk--Ulam type question.

\medskip
Let $d\geq 2$, $\ell\geq 1$, $k\geq 1$ and $j\geq 1$ be integers with $1\leq\ell\leq d-1$.
Consider a collection of $j$ mass assignments $\M=(\mu_1, \dots, \mu_j)$ on the Grassmann manifold $\Gr{\ell}{d}$. 
Our aim is to find a linear subspace $L \in \Gr{\ell}{d}$, and a $k$-arrangement $\HH^L=(H_1^L, \dots, H_k^L)$ in $L$, such that $\HH^L$ equiparts the collection of masses $\M^L=(\mu_1^L , \dots, \mu_j^L)$.

\medskip
In order to apply topological methods and derive an appropriate configuration test map scheme for our problem we make an additional assumption on the mass assignment $\mu_1$.
We assume that for every linear subspace $L \in \Gr{\ell}{d}$ the associated mass $\mu_1^L$ has compact and connected support.
In particular, this means that for every direction in $L$ there exists a unique oriented affine hyperplane orthogonal to the direction equiparting $\mu_1^L$.  
Or in other words, for fixed $L$ the space of all oriented affine hyperplanes in $L$ equiparting $\mu_1^L$ is homeomorphic to the sphere $S(L)\approx S^{\dim (L)-1}$.
This assumption will not affect the final results since we can strongly approximate each mass by masses with compact connected support.

\medskip
Now, in several steps we derive a configuration test map scheme which enables the application of advanced topological methods to the Gr\"unbaum--Hadwiger--Ramos problem for mass assignments.

\subsection{Configuration space}
\label{subsec : Configuration space}
For the configuration space associated to our problem, or in other words the space of all solution candidates, we take collections of $k$-arrangements in each of the linear spaces $L \in \Gr{\ell}{d}$ such that each of $k$ affine hyperplanes in $L$ equipart the mass $\mu_1^L$ (determined by the first mass assignment $\mu_1$).

\medskip
More precisely, let us first consider the tautological bundle on the Grassmann manifold $\Gr{\ell}{d}$:
\[
\xi := \gamma_{\ell}^{d} = \big(E(\gamma_{\ell}^{d}), \ B(\gamma_{\ell}^{d})=\Gr{\ell}{d},\  E(\gamma_{\ell}^{d})\xrightarrow{~\pi~} B(\gamma_{\ell}^{d}), \ F(\gamma_{\ell}^{d})=\R^{\ell} \big).
\]
The assumption on the mass assignment $\mu_1$ allows us to see the associated sphere bundle of $\xi$:
\[
S\xi := S\gamma_{\ell}^{d} = \\ 
\big(E(S\xi), \ B(S\xi)=\Gr{\ell}{d},\   E(S\xi)\xrightarrow{~S\pi~} B(S\xi), \ F(S\xi)=S^{\ell-1} \big),
\]
as the space of all oriented affine hyperplanes of linear subspaces $L\in\Gr{\ell}{d}$ which equipart the corresponding masses $\mu_1^L$.
In this way we have already obtained a configuration space for the case $k=1$.

\medskip
For $k\geq 2$, that is for arrangements with more than one hyperplane, we proceed as follows. 
We consider the $k$-fold product bundle  $(S\xi)^k$:
\[
(S\xi)^k = 
(E(S\xi)^k, \ \Gr{\ell}{d}^k,\  E(S\xi)^k\xrightarrow{q_k:=(S\pi)^k} \Gr{\ell}{d}^k, \ F((S\xi)^k)=(S^{\ell-1})^k ),
\]
and take the pullback along the diagonal embedding $\Delta_k\colon \Gr{\ell}{d}\longrightarrow \Gr{\ell}{d}^k$: 
\[ 
\xymatrixcolsep{1in}	
\xymatrix{
E\bigC{\Delta^*_k( (S\xi)^k)} \ar[r] \ar[d]^{p_k} & E(S\xi)^k   \ar[d]^{ q_k  }  \\
       \Gr{\ell}{d} \ar[r]^{\Delta_k} &  \Gr{\ell}{d}^k.
  }
\]  
The space of all solution candidates, associated to the parameters $(d,\ell,k)$, is the total space of the pullback bundle:
\begin{align*}
\CC(d,\ell,k) &:= 	E\bigC{\Delta^*_k( (S\xi)^k)}\\
&\ = \{ (L ; v_1, \ldots, v_k) \mid L \in \Gr{\ell}{d}, v_i \in L, \;  \|v_1\|=\cdots=\|v_k\| =1 \}.
\end{align*}
This means that $\CC(d,\ell,k)$ is the total space of the fibre bundle $\Delta^*_k( (S\xi)^k)$:
\begin{equation}
\label{bundle : configuration space}
\xymatrix{
(S^{\ell-1})^k\ar[r] & \CC(d,\ell,k)\ar[r]^-{p_k} & \Gr{\ell}{d},
}	
\end{equation}
where the map $p_k$ is given by $(L ; v_1, \ldots, v_k)\longmapsto L$.
Recall that by our assumption for every $ (L ; v_1, \ldots, v_k) \in \CC(d,\ell,k)$ each of the vectors $v_i$ parametrizes the oriented affine hyperplane in $L$ orthogonal to $v_i$, oriented by $v_i$, which equiparts the mass $\mu_1^L$.

\medskip
The Weyl group $\Sympm_k:=\Z_2^k \rtimes \Sym_k$, also called the group of signed permutations, acts fibrewise on $\CC(d,\ell,k)$ by
\[
((\beta_1, \dots, \beta_k) \rtimes \tau)\cdot (L ; v_1, \ldots, v_k) = (L ;  (-1)^{\beta_1} v_{\tau^{-1}(1)} , \ldots, (-1)^{\beta_k} v_{\tau^{-1}(k)}),
\]
for  $((\beta_1, \dots, \beta_k) \rtimes \tau) \in \Sympm_k$ and $(L ; v_1, \ldots, v_k)\in \CC(d,\ell,k)$.

\subsection{The test space}
\label{subsec : The test space}

Consider the real vector space $\R^{\Z_2^k}$ and its  codimension $1$ subspace:
\[ 
U_k:=\Big\{ (y_{\alpha})_{\alpha\in \Z_2^k} \in \R^{\Z_2^k} : \sum_{\alpha\in \Z_2^k} y_{\alpha} =0\Big\}.
\]
The structure of a real $\Sympm_k$-representation on the vectors space $\R^{\Z_2^k}$ is induced as follows. 
The element $((\beta_1, \dots, \beta_k) \rtimes \tau) \in \Sympm_k$ acts on the vector $(y_{(\alpha_1,\dots,\alpha_k)})_{(\alpha_1,\dots,\alpha_k) \in \Z_2^k}$ in $\R^{\Z_2^k}$ by permuting the indices:
\[
((\beta_1, \dots, \beta_k) \rtimes \tau) \cdot (\alpha_1, \dots, \alpha_k) = (\beta_1 + \alpha_{\tau^{-1}(1)}, \dots, \beta_k + \alpha_{\tau^{-1}(k)}).
\]
Here the addition is assumed to be in $\Z_2$. With respect to this action, the subspace $U_k$ is a $\Sympm_k$-subrepresentation of $\R^{\Z_2^k}$. 

\medskip
Let us first consider $\R^{\Z_2^k}$ as an $\Z_2^k$-representation where $\Z_2^k\subseteq \Z_2^k \rtimes \Sym_k=\Sympm_k$.
For $\omega=(\omega_1,\dots,\omega_k)\in \Z_2^k$ we denote by $V_{\omega}$ the $1$-dimensional real $\Z_2^k$-representation defined by:
\[
\alpha\cdot v:=(-1)^{\omega_1\alpha_1}\cdots (-1)^{\omega_k\alpha_k}v=(-1)^{\omega_1\alpha_1+\cdots +\omega_k\alpha_k}v
\]
for $\alpha=(\alpha_1,\dots,\alpha_k)$ and $v\in V_{\omega}$.
Then there are isomorphisms of $\Z_2^k$-representations:
\[
\R^{\Z_2^k}\cong \bigoplus_{\omega\in\Z_2^k}V_{\omega}
\qquad\text{and}\qquad
U_k\cong \bigoplus_{\omega\in\Z_2^k{\setminus}\{0\}}V_{\omega}.
\]

\medskip
In order to obtain a decomposition of $\R^{\Z_2^k}$, now as an $\Sympm_k$-representation, let us first partition $\Z_2^k$, as a set, into the disjoint union $\Z_2^k=A_0\sqcup   A_1\sqcup \dots\sqcup A_k$ where, for $0\leq i\leq k$, we define
\[
A_i:=\{ (\omega_1,\dots,\omega_k)\in \Z_2^k : \omega_1+\dots+\omega_k=i\}.
\]
Addition is assumed to be in $\Z$.
It is not hard to see that for every $0\leq i\leq k$ the direct sum $W_i:=\bigoplus_{\omega\in A_i}V_{\omega}$ is a $\Sympm_k$-representation, and in addition there are isomorphisms of $\Sympm_k$-representations:
\[
\R^{\Z_2^k}\cong W_0\oplus W_1\oplus\dots\oplus W_k
\qquad\text{and}\qquad
U_k\cong W_1\oplus\dots\oplus W_k.
\]
We set $U_k':= W_2\oplus\dots\oplus W_k$, and consequently $U_k\cong W_1\oplus U_k'$.

\medskip
The test space, associated to the parameters $(d,\ell,j,k)$, is the total space of the following trivial vector bundle $\pi_2$:
\begin{equation}
\label{bundle : test map}
\xymatrix{
U_k'\oplus U_k^{\oplus j-1}\ar[r] & (U_k'\oplus U_k^{\oplus j-1})\times \Gr{\ell}{d}\ar[r]^-{\pi_2} & \Gr{\ell}{d}.
}
\end{equation}
The reason for such a choice of the test space will become clear in the next section and Theorem \ref{thm : cs-tm scheme}.
The group $\Sympm_k$ acts on the product $(U_k'\oplus U_k^{\oplus j-1})\times \Gr{\ell}{d}$ diagonally where the action on the Grassmann manifold $\Gr{\ell}{d}$ is assumed to be trivial.

\subsection{The test map}
\label{subsec : The test map}

We recall that the parameters $(d,\ell,j,k)$ are fixed and in addition we have fixed a collection of $j$ mass assignments $\M=(\mu_1, \dots, \mu_j)$ on the Grassmann manifold $\Gr{\ell}{d}$. 
The test map associated to the collection $\M$ is the bundle map
\[
\phi_{\M} \colon \CC(d,\ell,k)\longrightarrow (U_k'\oplus U_k^{\oplus j-1})\times \Gr{\ell}{d}
\]
defined by
\begin{align*}
(L ; v_1, \ldots , v_k)	\longmapsto
  & 
\Big(L ;
\big(\mu^L_1(H^{\alpha_1}_{v_1} \cap \dots \cap H^{\alpha_k}_{v_k}) - \tfrac{1}{2^k}\mu^L_1(L) \big)_{(\alpha_1,\dots,\alpha_k)\in \Z_2^k} , \\
&\big( \big(\mu^L_i(H^{\alpha_1}_{v_1} \cap \dots \cap H^{\alpha_k}_{v_k}) - \tfrac{1}{2^k}\mu^L_i(L) \big)_{(\alpha_1,\dots,\alpha_k)\in\Z_2^k} \big)_{2\leq i\leq k  } \Big).
\end{align*}
The fact that 
\[
\big(\mu^L_1(H^{\alpha_1}_{v_1} \cap \dots \cap H^{\alpha_k}_{v_k}) - \tfrac{1}{2^k}\mu^L_1(L)  \big)_{(\alpha_1,\dots,\alpha_k)\in \Z_2^k} \ \in U_k'
\]
is a consequence of our assumption that for every $ (L ; v_1, \ldots, v_k) \in \CC(d,\ell,k)$ each of the vectors $v_i$ parametrizes the oriented affine hyperplane in $L$ orthogonal to $v_i$, oriented by $v_i$, which equiparts the mass $\mu_1^L$. 
 
 \medskip
 The test map $\phi_{\M}$ is a (fibrewise) $\Sympm_k$-equivariant map with respect to the already introduced actions of $\Sympm_k$ on the configuration space $\CC(d,\ell,k)$ and on the test space $(U_k'\oplus U_k^{\oplus j-1})\times \Gr{\ell}{d}$.
 The key property of the construction we have made is that {\em for the given collection of mass assignments $\M$ there exists a $k$ arrangement in the linear subspace $L\in \Gr{\ell}{d}$ which equiparts the collection of masses $\M^L$ if and only if $(L ; 0,0) \in \im(\phi_{\M})\cap \big( (U_k' \oplus U_k^{\oplus j-1})\times \Gr{\ell}{d}\big)$}.
 Consequently we have proved the following theorem.

 \begin{theorem}
 \label{thm : cs-tm scheme}	
 Let $d\geq 2$, $\ell\geq 1$, $k\geq 1$ and $j\geq $ be integers with $1\leq\ell\leq d-1$.
 Assume that $\Delta^*_k( (S\xi)^k)$ and $\pi_2$ are the fibre bundles we already introduced in \eqref{bundle : configuration space} and \eqref{bundle : test map}, and denote by $S\pi_2$ the associated sphere bundle of $\pi_2$.
 \begin{compactenum}[\rm \quad (a)]
 \item \label{thm : cs-tm scheme - a}	 	
 If $(d,\ell,j,k)$ is not mass assignment admissible, then there exists an $\Sympm_k$-equivariant bundle map $\Delta^*_k( (S\xi)^k)	\longrightarrow	S\pi_2$.
\item \label{thm : cs-tm scheme - b}	 	 
If there in no $\Sympm_k$-equivariant bundle map $\Delta^*_k( (S\xi)^k)\longrightarrow	S\pi_2$, then $(d,\ell,j,k)$ is mass assignment admissible.
 \end{compactenum}

 \end{theorem}

\section{The Fadell--Husseini index of the configuration and the test space}
\label{sec : The Fadell--Husseini index of the configuration and test space}

In this section, as the essential ingredient of the proof of Theorem \ref{theorem : Main}, we compute the Fadell--Husseini indices of the fibre bundles $\Delta^*_k( (S\xi)^k)$ and $S\pi_2$ with respect to the action of the subgroup $\Z_2^k$ of the group of signed permutations $\Sympm_k$.
For the definition of the fibre bundles recall \eqref{bundle : configuration space} and \eqref{bundle : test map}.
First, we recall what is the Fadell--Husseini index and collect all the necessary tools for the computations which will follow.

\subsection{The Fadell--Husseini index}
\label{subsec : The Fadell--Husseini index}

In $1988$ Edward Fadell and Sufian Husseini, in their seminal paper \cite{Fadell-Husseini-1988}, introduced a notion of the ideal-valued index theory, a covariant functor $\ind_G(\, \underline{\hspace{0.25cm}}\, ;A)$ from the category of topological $G$-spaces into the partially ordered set, seen as a category, of all ideals in the cohomology ring $H^*(\B G;A)$ ordered by inclusion. 
Here $A$ denotes a commutative ring with $1$.  
This means in particular that if $X$ and $Y$ are $G$-spaces and there is a continuous $G$-equivariant map $X\longrightarrow Y$, then  $\ind_G(X;A)\supseteq \ind_G(Y;A)$.

\medskip
In this paper we use a slight extension of the original notion of the ideal-valued index theory to the category of all continuous $G$-equivariant maps from $G$-spaces to the fixed space $B$ equipped with the trivial $G$-action.
More precisely, let $B$ be a fixed topological space with the trivial $G$-action, that is $g\cdot b=b$ for all $b\in B$ and all $g\in G$.
Consider in addition a continuous $G$-equivariant map $p \colon X \longrightarrow B$, which in this case means that $p(g\cdot x)=g\cdot p(x)=p(x)$ for every $x\in X$ and every $g\in G$.
The \emph{Fadell-Husseini index} of $p$ with coefficients in a commutative ring $A$ with identity is defined to be the kernel ideal of the following induced map 
\begin{align*}
\Index{p}{B}{G}{A} & := \ker \big((\id\times_G p)^* \colon H^*(\E G \times_G B;A) \longrightarrow H^* (\E G \times_G X;A)  \big)  \\
&\ = \ker \big( (\id\times_G p)^* \colon H_G^*(B;A) \longrightarrow H_G^* (X;A)  \big).
\end{align*}
Here $H^*(\cdot)$ denotes the \v{C}ech cohomology while $H^*_G(\cdot)$ stands for the equivariant cohomology defined as the \v{C}ech cohomology of the Borel construction.
The essential properties of the index are:
\begin{compactenum}[\rm (i)]
\item \emph{Monotonicity:} If $p \colon X \longrightarrow B$ and $q \colon Y \longrightarrow B$ are continuous $G$-equivariant maps, and $f \colon X \longrightarrow Y$ is a continuous $G$-equivariant map with the property that $p = q \circ f$, then 
\[ 
\Index{p}{B}{G}{A} \supseteq \Index{q}{B}{G}{A}.  
\]	
\item \emph{Additivity:} If $(X_1 \cup X_2, X_1, X_2)$ is an excisive triple of $G$-spaces and $p \colon X_1 \cup X_2 \longrightarrow B $ is a continuous $G$-equivariant map, then
\[ 
\Index{p|_{X_1}}{B}{G}{A} \cdot \Index{p|_{X_2}}{B}{G}{A}  \subseteq \Index{p}{B}{G}{A}.
 \]
\item \emph{General Borsuk-Ulam-Bourgin-Yang theorem:} Let $p \colon X \longrightarrow B$ and $q \colon Y \longrightarrow B$ be continuous $G$-equivariant maps, and let $f \colon X \longrightarrow Y$ be a continuous $G$-equivariant map such that $p = q \circ f$. If $Z$ is a $G$-invariant subspace of $Y$, then 
\[ 
\Index{p|_{f^{-1}(Z)}}{B}{G}{A} \cdot \Index{q|_{Y \setminus Z}}{B}{G}{A} \subseteq \Index{p}{B}{G}{A}. 
\]

\end{compactenum} 
In the case when the space $B$ is a point and $p \colon X \longrightarrow B$ is, consequently, just the constant map, we recover the original definition of the ideal-valued index of the $G$-space $X$.
For that reason we simplify the notation and write $\Index{p}{B}{G}{A}=\Index{X}{\pt}{G}{A}=\ind_G(X;A)$. 
In this notation, the next property of the index can be formulated as follows: 

\medskip
\begin{compactenum}[\rm (iv)]
\item If $X$ is a $G$-space and $p \colon B \times X \longrightarrow B$ is the projection on the first factor, then 
\[ \Index{p}{B}{G}{A}= \Index{X}{\pt}{G}{A} \otimes H^*(B;A). \]
	
\end{compactenum}

\subsection{The Fadell--Husseini index of the test space}
\label{subsec : The Fadell--Husseini index of the test space}

In this section we compute the Fadell--Husseini index of the following projection onto the second factor
\[
\pi_2 \colon S(U_k'\oplus (U_k)^{\oplus j-1})\times \Gr{\ell}{d}\longrightarrow \Gr{\ell}{d}
\]
with respect to the action of the subgroup $\Z_2^k\subseteq\Sympm_k$ and with coefficients in the field with two elements $\F_2$.
That is: 
\begin{multline*}
\Index{\pi_2}{\Gr{\ell}{d}}{\Z_2^k}{\F_2}:=  \\	
\ker
\Big(
 (\id \times_{\Z_2^k}\pi_2)^* \colon H^*\big(\E\Z_2^k\times_{\Z_2^k} \Gr{\ell}{d};\F_2\big) \longrightarrow \\H^*\big(\E\Z_2^k\times_{\Z_2^k}	(S(U_k'\oplus (U_k)^{\oplus j-1}\times \Gr{\ell}{d});\F_2\big).
 \Big)
\end{multline*}
Since $\pi_2$ is a trivial fibre bundle and the $\Z_2^k$-action on $\Gr{\ell}{d}$ is trivial we see that induced map 
\[
\id \times_{\Z_2^k}\pi_2\colon \E\Z_2^k\times_{\Z_2^k}\big(	S(U_k'\oplus (U_k)^{\oplus j-1})\times \Gr{\ell}{d}\big)\longrightarrow 
\E\Z_2^k\times_{\Z_2^k} \Gr{\ell}{d}
\]
can be transformed into the following product map
\[
	\id \times_{\Z_2^k}\pi_2=u\times\id \colon \big(\E\Z_2^k\times_{\Z_2^k}S(U_k'\oplus (U_k)^{\oplus j-1}) \big)\times \Gr{\ell}{d}\longrightarrow \B\Z_2^k\times \Gr{\ell}{d}.
\]
Here, the map $u\colon \E\Z_2^k\times_{\Z_2^k}S(U_k'\oplus (U_k)^{\oplus j-1})\longrightarrow\B\Z_2^k$ is induced, via the Borel construction, by the constant $G$-equivariant map $S(U_k'\oplus (U_k)^{\oplus j-1})\longrightarrow\pt$.
In particular,
\begin{multline*}
\Index{S(U_k'\oplus (U_k)^{\oplus j-1})}{\pt}{\Z_2^k}{\F_2}=\\
\ker
\big(
u^*\colon H^*(\B\Z_2^k;\F_2)\longrightarrow H^*(\E\Z_2^k\times_{\Z_2^k} S(U_k'\oplus (U_k)^{\oplus j-1});\F_2)
\big).
\end{multline*}
Consequently, the induced map in cohomology $(\id \times_{\Z_2^k}\pi_2)^*=(u\times\id)^*$, after an application of the K\"unneth formula for field coefficients\cite[Thm.\,VI.3.2]{Bredon-1997}, becomes the tensor product homomorphism
\begin{multline*}
u^*\otimes\id\colon H^*(\B\Z_2^k;\F_2)\otimes H^*(\Gr{\ell}{d};\F_2)\longrightarrow  \\
 H^*(\E\Z_2^k\times_{\Z_2^k} S(U_k'\oplus (U_k)^{\oplus j-1});\F_2)\otimes H^*(\Gr{\ell}{d};\F_2).
\end{multline*}
Thus,
\begin{multline*}
\Index{\pi_2}{\Gr{\ell}{d}}{\Z_2^k}{\F_2}=\ker ((\id \times_{\Z_2^k}\pi_2)^*)=\ker (u^*\otimes\id)
=\\
\Index{S(U_k'\oplus (U_k)^{\oplus j-1})}{\pt}{\Z_2^k}{\F_2}\otimes  H^*(\Gr{\ell}{d};\F_2).
\end{multline*}
Now, if we use \cite[Prop.\,12]{Blagojevic-Ziegler-2011} we conclude that
\begin{multline}\label{eq : index of the test space}
\Index{\pi_2}{\Gr{\ell}{d}}{\Z_2^k}{\F_2}=\\
 \Big\langle 
\frac{1}{x_1\cdots x_k}\prod_{(\alpha_{1},\dots, \alpha_{k})\in\F_2^k{\setminus}\{ (0,\ldots,0) \}} (\alpha_{1}x_1 + \cdots + \alpha_{k}x_k)^{j}
  \Big\rangle \otimes H^*(\Gr{\ell}{d};\F_2).	
\end{multline}
Here we use the classical fact that the cohomology ring of $\B\Z_2^k$ can be presented as the polynomial ring:
\[
H^*(\B\Z_2^k;\F_2)\cong\F_2[x_1,\dots,x_k],
\] 
where all generators are of degree $1$, that is $\deg(x_1)=\dots=\deg(x_k)=1$, consult for example \cite[Thm.\,II.4.4]{Adem-Milgram-2004}.

\medskip 
For the cohomology of the real Grassmann manifold recall the classical result of Borel \cite[p. 190]{Borel-1953} which gives a presentation of the cohomology with $\F_2$ coefficients in the form of the truncated polynomial ring:
\[ 
H^*(\Gr{\ell}{d};\F_2) \cong \F_2[ w_1, \dots ,w_{\ell}, \overline{w}_1, \dots , \overline{w}_{d-\ell} ] / I_{d,\ell},
\]
where $\deg( w_i) = i$, $\deg( \overline{w}_j) = j$ for $1\leq i\leq \ell$, $1\leq j\leq d-\ell$, and $I_{d,\ell}$ is the ideal generated by the $d$ relations of the equality
\[ 
(1 + w_1 + \dots + w_{\ell}) (1 + \overline{w}_1 + \dots + \overline{w}_{d-\ell}) = 1. 
\]
In other word, the ideal $I_{d,\ell}$ is generated by the polynomials:
\[
\sum_{j=\max\{0,s+\ell-d\}}^{\min\{s,\ell\}}w_j\cdot \overline{w}_{s-j}, \qquad 1\leq s\leq d.
\]
Here the generator $w_i$, for $ 1 \leq i \leq \ell $, can be identified with the Stiefel--Whitney classes of the canonical bundle $\gamma_{\ell}^d$ while the remaining generators $\overline{w}_i$, for $ 1 \leq i \leq d-\ell $, can be identified with the dual Stiefel--Whitney classes of $\gamma_{\ell}^d$.
 
 \medskip

\subsection{The Fadell--Husseini index of the configuration space}
\label{subsec : The Fadell--Husseini index of the configuration space}

In Section \ref{subsec : Configuration space} we have defined the configuration space as the total space of the pull-back bundle
\[ 
\xymatrixcolsep{1in}	
\xymatrix{
\CC(d,\ell,k)=E\bigC{\Delta^*_k( (S\xi)^k)} \ar[r] \ar[d] & E(S\xi)^k   \ar[d]^{ q_k  }  \\
       \Gr{\ell}{d} \ar[r]^{\Delta_k} &  \Gr{\ell}{d}^k.
  }
\]  
where $\Delta_k$ is the diagonal embedding.
More precisely, 
\begin{align*}
\CC(d,\ell,k) &= 	E\bigC{\Delta^*_k( (S\xi)^k)}\\
&= \{ (L ; v_1, \ldots, v_k) \mid L \in \Gr{\ell}{d}, v_i \in L, \;  \|v_1\|=\cdots=\|v_k\| =1 \}.
\end{align*}
In this section we determine the Fadell--Husseini index of the map
\[
p_k \colon\CC(d,\ell,k)\longrightarrow \Gr{\ell}{d}
\]
with respect to the action of the subgroup $\Z_2^k\subseteq\Sympm_k$ and with coefficients in the field with two elements $\F_2$.
In other words, we describe:  
\begin{multline*}
\Index{p_k}{\Gr{\ell}{d}}{\Z_2^k}{\F_2}:=  \\	
\ker
\Big(
 (\id \times_{\Z_2^k}p_k)^* \colon H^*\big(\E\Z_2^k\times_{\Z_2^k} \Gr{\ell}{d};\F_2\big) \longrightarrow \\H^*\big(\E\Z_2^k\times_{\Z_2^k}	\CC(d,\ell,k);\F_2\big)
 \Big).
\end{multline*}
The computation of 
$
\Index{p_k}{\Gr{\ell}{d}}{\Z_2^k}{\F_2}\subseteq H^*\big(\E\Z_2^k\times_{\Z_2^k} \Gr{\ell}{d};\F_2\big)
$
is done in two steps. 
First, we describe the index 
\[
\Index{q_k}{\Gr{\ell}{d}^k}{\Z_2^k}{\F_2}\subseteq H^*\big(\E\Z_2^k\times_{\Z_2^k} \Gr{\ell}{d}^k;\F_2\big),
\] 
and then show that 
\[
\Index{p_k}{\Gr{\ell}{d}}{\Z_2^k}{\F_2}=(\Delta_k)^*\big(\Index{q_k}{\Gr{\ell}{d}^k}{\Z_2^k}{\F_2}\big).
\]

\subsubsection{Index of $q_k$}
\label{subsubsec : Index of q_k}

The Fadell--Husseini index of the map
$
q_k\colon  E(S\xi)^k\longrightarrow \Gr{\ell}{d}^k
$ 
is, by definition, the kernel: 
\begin{multline*}
\Index{q_k}{\Gr{\ell}{d}}{\Z_2^k}{\F_2}: = \\
\ker
\Big(
(\id \times_{\Z_2^k}q_k)^* \colon H^*\big(\E\Z_2^k\times_{\Z_2^k} \Gr{\ell}{d}^k;\F_2\big) \longrightarrow \\ H^*\big(\E\Z_2^k\times_{\Z_2^k}	 E(S\xi)^k;\F_2\big)
 \Big).
\end{multline*}
In order to identify the index $\Index{q_k}{\Gr{\ell}{d}}{\Z_2^k}{\F_2}$ we consider the fibre bundle
\[
\xymatrix{
(S^{\ell-1})^k\ar[r]& \E\Z_2^k\times_{\Z_2^k} E( S\xi)^k \ar[r]& \E\Z_2^k\times_{\Z_2^k} \Gr{\ell}{d}^k= \B\Z_2^k \times \Gr{\ell}{d}^k.
}
\]
The associated Serre spectral sequence has $E_2$-term given by 
\[
	E_2^{i,j} =  H^{i}( \B\Z_2^k \times \Gr{\ell}{d}^k ;\mathcal{H}^{j}( (S^{\ell-1})^k ; \F_2) ),
\]
where the local coefficient system is determined by the action of the fundamental group of the base space. 
Since the fundamental group
\[
	\pi_1( \B\Z_2^k \times \Gr{\ell}{d}^k )  \cong  \pi_1( B\Z_2^k ) \times \pi_1( \Gr{\ell}{d} )^k 
		    										    	\cong  \Z_2^k \times  \Z_2^{\oplus k}
\]
acts trivially on the cohomology of the fibre ${H}^{j}( (S^{\ell-1})^k ; \F_2)$, the $E_2$-term of the spectral sequence simplifies and becomes 
\begin{equation}
\label{Sseq1}	
E_2^{i,j}\cong H^{i}(\Z_2^k \times \Gr{\ell}{d}^k ; \F_2 ) \otimes  H^{j}( ( S^{\ell-1} )^k ; \F_2 ).
\end{equation}
In addition, all the differentials of the spectral sequence satisfy the Leibniz rule.

\medskip
Let us denote the cohomology of the fibre $( S^{\ell-1} )^k$ as follows 
\begin{multline*}
H^*( (S^{\ell-1})^k ;\F_2 )\cong H^*( S^{\ell-1} ;\F_2 ) ^{\otimes k} \cong \F_2[y_1]/(y_1^2) \otimes \cdots \otimes \F_2[y_k]/(y_k^2) \cong \\\F_2[y_1, \cdots , y_k]/(y_1^2, \cdots , y_k^2)	
\end{multline*}
 where $\deg(y_1)=\dots=\deg(y_k) = \ell-1$.

\begin{figure}[h]
	\centering
	\includegraphics[scale=.35]{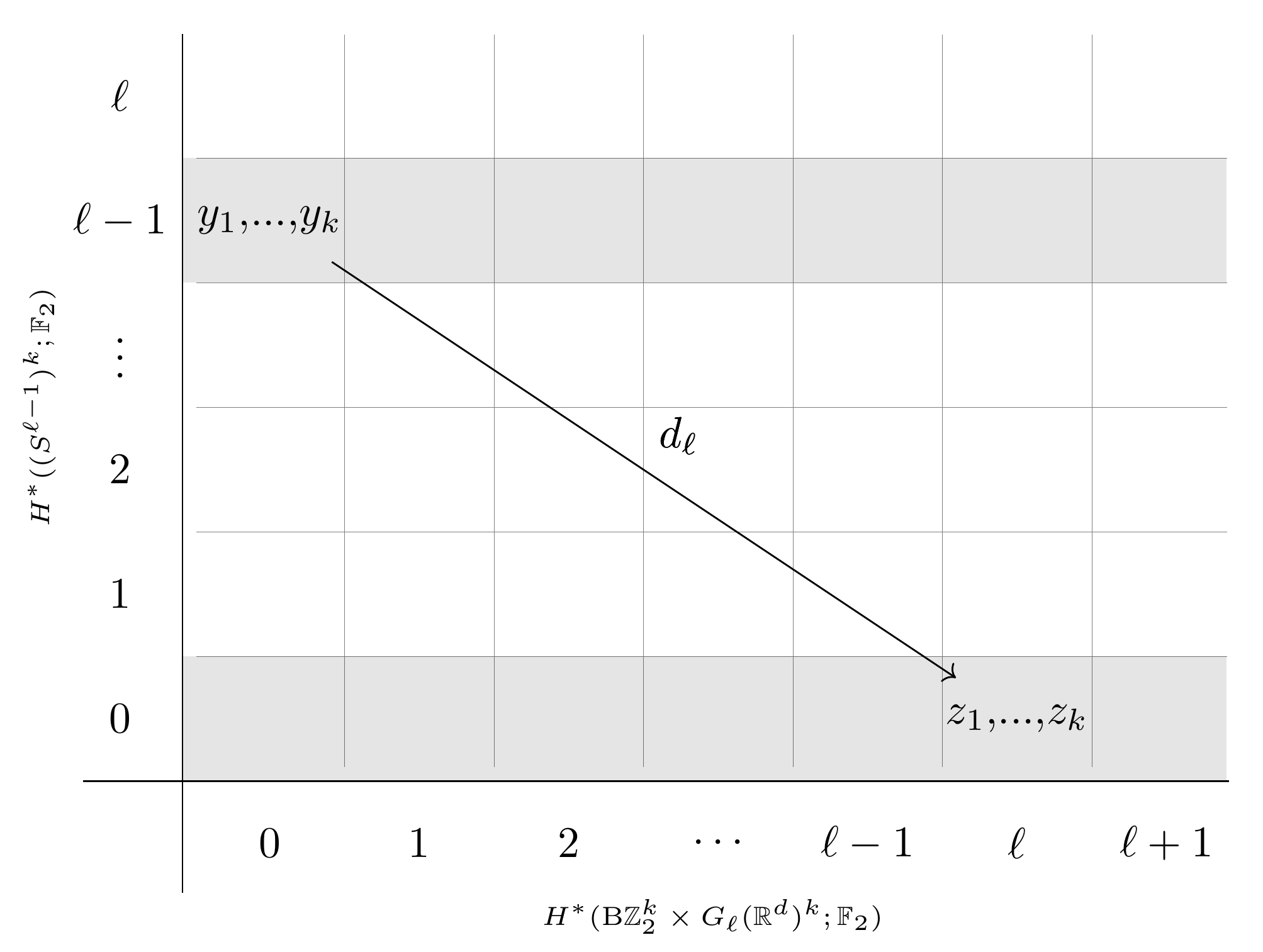}
	\caption{The $E_{\ell}$-term of the spectral sequence \eqref{Sseq1}.}
	\label{Figure: sseq 1}
\end{figure}

\medskip
Now we determine the first possible non-trivial differential $d_{\ell}$ by finding its values on the generators of the cohomology ring of the fibre, these are $z_1=d_{\ell}(y_1), \dots , z_k=d_{\ell}(y_k)$.
For that, consider the morphism of bundles 
\[ \xymatrixcolsep{1in}
	\xymatrix{
	E( S\xi)^k   \ar[r]\ar[d]^{q_k} & E( S\xi ) \ar[d] \\
	\Gr{\ell}{d}^k  \ar[r]^{g_r} & \Gr{\ell}{d} 
}\]
induced by the projection $g_r\colon \Gr{\ell}{d}^k\longrightarrow\Gr{\ell}{d}$ of the $r$th factor,  for $1 \leq r \leq k$. 
This morphism induces yet another morphism of the bundles
\[ \xymatrixcolsep{1in}
	\xymatrix{
	 \E\Z_2^k\times_{\Z_2^k}E( S\xi)^k  \ar[r]\ar[d] & \E\Z_2^k\times_{\Z_2^k}E( S\xi) \ar[d] \\
	\B\Z_2^k \times \Gr{\ell}{d}^k \ar[r]^{\id \times g_r} & \B\Z_2^k \times \Gr{\ell}{d},
}\label{MSseq1}\]
because it respects the action of $\Z_2^k$. 
The new morphism between Borel constructions, in turn, induces a morphism of the corresponding Serre spectral sequences which on the level of fibres, i.e., the zero column, is a monomorphism. 
Furthermore, it is also a monomorphism on the zero row of the $E_2$ and consequently $E_{\ell}$-term.
Thus,
\[ z_{r} = (\id \times g_r)^*(z),  \]
and so we turn our attention to the Serre spectral sequence associated to the fibre bundle 
\[ 
\xymatrix{
S^{\ell-1}\ar[r] & \E\Z_2^k\times_{\Z_2^k} E( S\xi )  \ar[r] & \B\Z_2^k\times \Gr{\ell}{d} .
}
\]
In particular, we want to determine $z=d_{\ell}(y)$, where $y \in H^*(  S^{\ell-1} ;\F_2 )$ is the generator.
For an illustration of this morphism of spectral sequences see Figure \ref{fig:morphism}.

\begin{figure}[H]
	\centering
	\includegraphics[scale=.60]{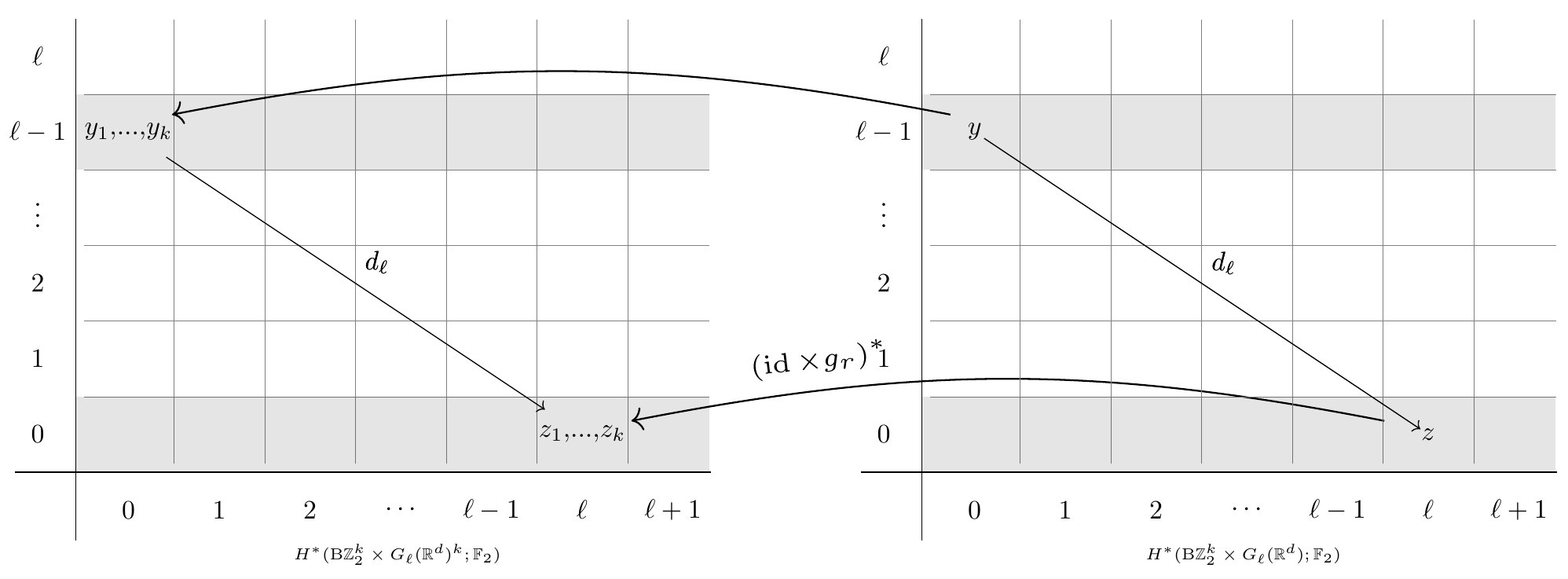}
	\caption{ Morphism of spectral sequences induced by the projection $g_r$.}
	\label{fig:morphism}
\end{figure}

\medskip
Now consider the inclusion $\Z_2 \xrightarrow{ i_r } \Z_2^k$ into the $r$th summand. 
Then there is an induced map between corresponding classifying spaces 
\[ 
\B\Z_2 \xrightarrow{ B(i_r) } \B\Z_2^k,
\]
which further on induces a map between corresponding Borel construction bundles 
\[ \xymatrixcolsep{1in}
	\xymatrix{
	\E\Z_2^k\times_{\Z_2^k} E(S\xi) \ar[d] & \E\Z_2\times_{\Z_2} E(S\xi)   \ar[l]\ar[d] \\
	\B\Z_2^k \times \Gr{\ell}{d}  & \B\Z_2 \times \Gr{\ell}{d} \ar[l]_{\B(i_r) \times \id} .
}\]
This morphism  of bundles induces  a morphism of the corresponding Serre spectral sequences which on the zero column of the $E_2$-term is an isomorphism; for an illustration see Figure \ref{fig:morphism2}.
From the classical work of Albrecht Dold \cite[Sec.\,3 and Sec.\,4]{Dold-1988}, applied to the Serre spectral sequence associated with the fibre bundle
\[
\xymatrix
{
S^{\ell-1}\ar[r] &  \E\Z_2\times_{\Z_2} E(S\xi) \ar[r] & \B\Z_2 \times \Gr{\ell}{d}
,}
\]
we have that
\begin{multline*}
	 \overline{z}: = d_{\ell}(\overline{y}) = 
	 \sum_{j = 0}^{\ell}  x^{j} \otimes w_{\ell-j}(\xi) \\
	 \in H^*(\B\Z_2 \times \Gr{\ell}{d} ;\F_2) \cong H^*(\B\Z_2 ;\F_2) \otimes H^*( \Gr{\ell}{d} ;\F_2),
\end{multline*}
where 
\begin{compactitem}[\rm  --]
\item $\overline{y}\in H^{\ell-1}(S^{\ell-1};\F_2)\cong E_2^{0,\ell-1}\cong E_{\ell}^{0,\ell-1}\cong\F_2$ is the generator,
\item $\overline{z}\in  H^{\ell}(\B\Z_2 \times \Gr{\ell}{d} ;\F_2)\cong \bigoplus_{a=0}^{\ell}H^{a}(\B\Z_2 ;\F_2) \otimes H^{\ell-a}( \Gr{\ell}{d} ;\F_2)$,
\item$H^*(\B\Z_2;\F_2) = \F_2[ x ]$, with $\deg(x)=1$, and 
\item  $w_{0}(\xi),\dots, w_{\ell}(\xi)$ are the Stiefel--Whitney classes of the tautological vector bundle $\gamma_{\ell}^{d}$.
\end{compactitem}

\noindent
Consequently,
\[ 
z = \sum_{j = 0}^{\ell}  x_{r}^{j} \otimes w_{\ell-j}(\xi) \in  H^*(\B\Z_2^k ;\F_2) \otimes H^*( \Gr{\ell}{d} ;\F_2).
\]
Recall that we have already fixed the notation $H^*(\B\Z_2^k;\F_2) = \F_2[ x_1, \cdots , x_k ]$ where $\deg(x_1)=\cdots=\deg(x_k)=1$. 
Furthermore, for  $1\leq r\leq k$ we get that
\begin{multline*}
z_{r} = \sum_{j = 0}^{\ell}  x_{r}^{j} \otimes (1 \otimes \cdots \otimes  w_{\ell-j}(\xi) \otimes \cdots \otimes 1)	\ 
\in \  H^*(\B\Z_2^k ;\F_2 ) \otimes H^*( \Gr{\ell}{d} ;\F_2)^{\otimes k}.
\end{multline*}
If, for $0\leq i\leq \ell$ and $1\leq r\leq k$, we set
$
w_{i,r} := 1 \otimes \cdots \otimes  w_{i}(\xi) \otimes \cdots \otimes 1
$, then we can rewrite
\[ 
z_{r} = \sum_{j = 0}^{\ell}  x_{r}^{j} \otimes w_{\ell-j,r}. \]
This means that
\begin{equation}
	\label{eq : index of q_k}
\Index{q_k}{\Gr{\ell}{d}^k}{\Z_2^k}{\F_2} = \langle z_1, \ldots , z_k \rangle =  \Big\langle \sum_{j = 0}^{\ell}  x_{r}^{j} \otimes w_{\ell-j,r} 	\: : \:	1 \leq r \leq k \Big\rangle.
\end{equation}

\begin{figure}[H]
	\centering
	\includegraphics[scale=.62]{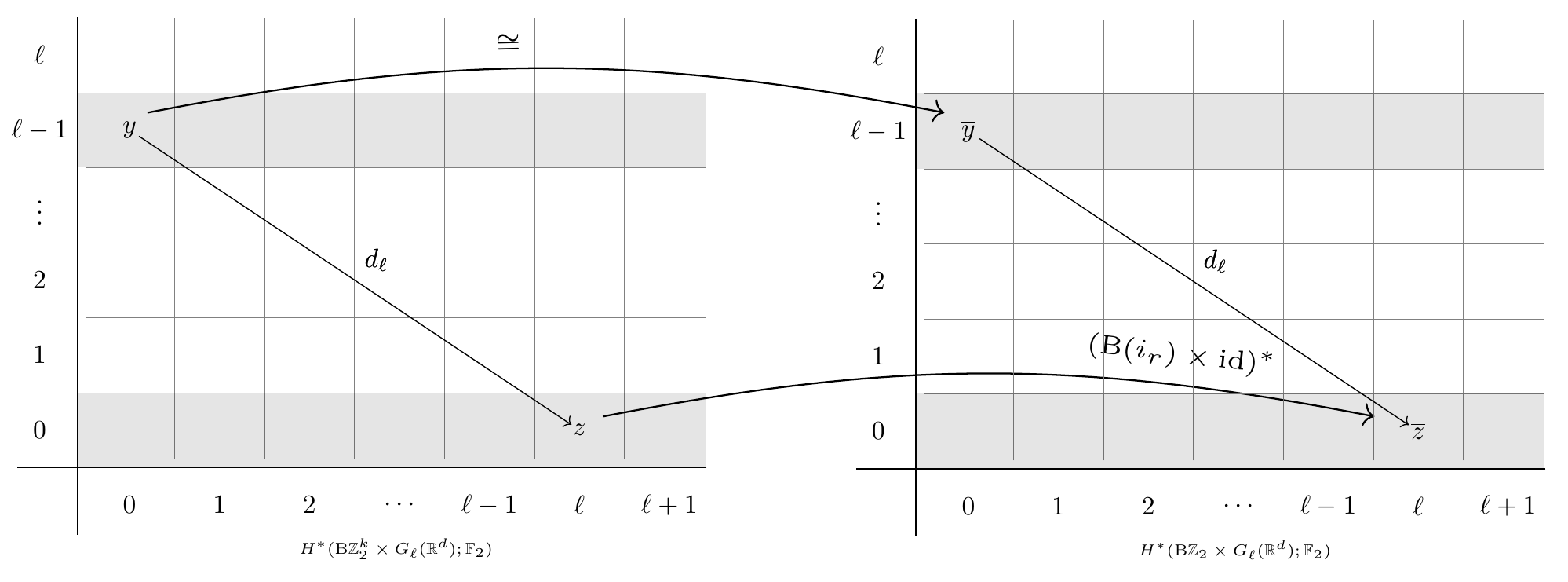}
	\caption{Morphism between spectral sequences induced by $\B(i_r)$.}
	\label{fig:morphism2}
\end{figure}

\subsubsection{Index of $p_k$}
\label{subsubsec : Index of p_k}
In this part we finally retrieve the Fadell--Husseini index of the map
$
p_k \colon\CC(d,\ell,k)\longrightarrow \Gr{\ell}{d}
$
with respect to the action of $\Z_2^k$ and with coefficients in  $\F_2$.
In other words, we compute:  
\begin{multline*}
\Index{p_k}{\Gr{\ell}{d}}{\Z_2^k}{\F_2}:=  \\	
\ker
\Big(
 (\id \times_{\Z_2^k}p_k)^* \colon H^*\big(\E\Z_2^k\times_{\Z_2^k} \Gr{\ell}{d};\F_2\big) \longrightarrow \\H^*\big(\E\Z_2^k\times_{\Z_2^k}	\CC(d,\ell,k);\F_2\big)
 \Big).
\end{multline*}

\medskip

The pullback diagram  
\[ 
\xymatrixcolsep{1in}	
\xymatrix{
\CC(d,\ell,k)=E\bigC{\Delta^*_k( (S\xi)^k)} \ar[r] \ar[d]^{ p_k  } & E(S\xi)^k   \ar[d]^{ q_k  }  \\
       \Gr{\ell}{d} \ar[r]^{\Delta_k} &  \Gr{\ell}{d}^k,
       }	
\] 
after applying the Borel construction, yields the following pullback diagram:
\[ 
\xymatrixcolsep{1in}	
\xymatrix{
\E\Z_2^k\times_{\Z_2^k}\CC(d,\ell,k)=\E\Z_2^k\times_{\Z_2^k}E\bigC{\Delta^*_k( (S\xi)^k)} \ar[r] \ar[d]^{ \id\times_{\Z_2^k}p_k  } & \E\Z_2^k\times_{\Z_2^k}E(S\xi)^k   \ar[d]^{ \id\times_{\Z_2^k}q_k  }  \\
    \B\Z_2^k\times   \Gr{\ell}{d} \ar[r]^{\Delta_k} &   \B\Z_2^k\times \Gr{\ell}{d}^k .
       }	
\] 
Furthermore, this morphism between the fibre bundles
\[
\xymatrix{
(S^{\ell-1})^k\ar[r]&\E\Z_2^k\times_{\Z_2^k}\CC(d,\ell,k)\ar[r] &  \B\Z_2^k\times   \Gr{\ell}{d}
}
\]
and
\[
\xymatrix{
(S^{\ell-1})^k\ar[r] & \E\Z_2^k\times_{\Z_2^k}E(S\xi)^k\ar[r] &  \B\Z_2^k\times   \Gr{\ell}{d}
}
\]
induces a morphism between their Serre spectral sequences.
Since the fibres of the bundles are homeomorphic, and the (non-trivial) fundamental groups of both base spaces act trivially on the cohomology of the fibres, the $E_2$-terms of the  spectral sequences are as follows:
\[
E_2^{i,j}( \id \times_{\Z_2^k} p_k ) \cong  H^{i}( \B\Z_2^k \times \Gr{\ell}{d}; \F_2 ) \otimes  H^{j}( ( S^{\ell-1} )^k ; \F_2 ) 
\]
and
\[
E_2^{i,j}( \id \times_{\Z_2^k} q_k ) \cong  H^{i}( B\Z_2^k\times \Gr{\ell}{d}^k ; \F_2 ) \otimes  H^{j}( ( S^{\ell-1} )^k ; \F_2 ) .
\]
Having in mind that the induced morphism of spectral sequences is an isomorphism on the $0$-column of the $E_2$-term, and that the differentials of both sequences satisfy the Leibniz rule, we get 
\begin{eqnarray*}
	\Index{p_k}{\Gr{\ell}{d}}{\Z_2^k}{\F_2}  & = & (\id \times \Delta_k)^*\big( \Index{q_k}{\Gr{\ell}{d}^k}{\Z_2^k}{\F_2} \big)\\
	& = & \Big\langle \sum_{s = 0}^{\ell}   (\id \times \Delta_k)^*( x_{r}^{s} \otimes w_{\ell-s,r} ) \: : \:	1 \leq r \leq k \Big\rangle \\
	& = & \Big\langle \sum_{s = 0}^{\ell}  x_{r}^{s} \otimes \Delta_k^*(w_{\ell-s,r}) 	\: : \:	1 \leq r \leq k \Big\rangle \\
	& = & \Big\langle \sum_{s = 0}^{\ell}  x_{r}^{s} \otimes w_{\ell-s} 	\: : \:	1 \leq r \leq k \Big\rangle.
\end{eqnarray*}
In summary, we have obtained that
\begin{equation}
\label{eq : index of the configuration space}
	\Index{p_k}{\Gr{\ell}{d}}{\Z_2^k}{\F_2} =	\Big\langle \sum_{s = 0}^{\ell}  x_{r}^{s} \otimes w_{\ell-s} 	\: : \:	1 \leq r \leq k \Big\rangle.
\end{equation}

\section{Proofs}
\label{sec : Proofs}

\subsection{Proof of Theorem \ref{theorem : Main}}
\label{Proof of Main Theorem}
Let $(d,\ell,k,j)$ be a $4$-tuple of natural numbers where $1\leq\ell\leq d-1$.
As introduced in \eqref{eq: definition of R}, we denote by $R_{d, \ell, k}$ the truncated polynomial ring
\[
\F_2[x_1,\dots,x_k, w_1, \dots ,w_{\ell}, \overline{w}_1, \dots , \overline{w}_{d-\ell}]/I_{d,\ell}
\]
where $\deg(x_1)=\dots=\deg(x_k)=1$, $\deg( w_i) = i$, $\deg( \overline{w}_j) = j$ for $1\leq i\leq \ell$, $1\leq j\leq d-\ell$, and $I_{d,\ell}$ is the ideal generated by the polynomials
\[
\sum_{j=\max\{0,s+\ell-d\}}^{\min\{s,\ell\}}w_j\cdot \overline{w}_{s-j}, \qquad 1\leq s\leq d.
\]
In addition, assume that 
\begin{multline}
\label{eq : assumption in main theorem}
\frac{1}{x_1\cdots x_k}\prod_{(\alpha_{1},\dots, \alpha_{k})\in\F_2^k{\setminus}\{ (0,\ldots,0) \}} (\alpha_{1}x_1 + \cdots + \alpha_{k}x_k)^{j}
 \\ \notin \
\Big\langle \sum_{s = 0}^{\ell}  x_{r}^{s} \otimes w_{\ell-s} 	\: : \:	1 \leq r \leq k \Big\rangle.	
\end{multline}
Now, we will prove that the $4$-tuple $(d,\ell,k,j)$ is mass assignment admissible.

\medskip
Let us assume the opposite, that $(d,\ell,k,j)$ is not mass assignment admissible.
From the configuration space / test map scheme, Theorem \ref{thm : cs-tm scheme}\eqref{thm : cs-tm scheme - a}, we get an  $\Sympm_k$-equivariant bundle map $\Delta^*_k( (S\xi)^k)	\longrightarrow	S\pi_2$.	
Here the bundles $\Delta^*_k( (S\xi)^k)$ and $\pi_2$ are defined in \eqref{bundle : configuration space} and \eqref{bundle : test map} respectively.
Thus, according to the monotonicity property of the Fadell--Husseini index, the following inclusion must hold:
\[
\Index{\pi_2}{\Gr{\ell}{d}}{\Z_2^k}{\F_2} \ \subseteq \ \Index{p_k}{\Gr{\ell}{d}}{\Z_2^k}{\F_2}.
\]
On the other hand, from \eqref{eq : index of the test space} and \eqref{eq : index of the configuration space}, we have that 
\begin{multline*}
\Index{\pi_2}{\Gr{\ell}{d}}{\Z_2^k}{\F_2}=\\
 \Big\langle \prod_{i=1}^{k} x_i^{j-1} \cdot \prod_{(\alpha_{1},\dots, \alpha_{k})\in\F_2^k{\setminus}\Gamma} (\alpha_{1}x_1 + \cdots + \alpha_{k}x_k)^{j} \Big\rangle \otimes H^*(\Gr{\ell}{d};\F_2),	
\end{multline*}
and
\[
	\Index{p_k}{\Gr{\ell}{d}}{\Z_2^k}{\F_2} =	\Big\langle \sum_{s = 0}^{\ell}  x_{r}^{s} \otimes w_{\ell-s} 	\: : \:	1 \leq r \leq k \Big\rangle.
\]
Consequently,
\begin{multline*}
 \Big\langle \prod_{i=1}^{k} x_i^{j-1} \cdot \prod_{(\alpha_{1},\dots, \alpha_{k})\in\F_2^k{\setminus}\Gamma} (\alpha_{1}x_1 + \cdots + \alpha_{k}x_k)^{j} \Big\rangle \otimes H^*(\Gr{\ell}{d};\F_2)  \ \subseteq \\
\Big\langle \sum_{s = 0}^{\ell}  x_{r}^{s} \otimes w_{\ell-s} 	\: : \:	1 \leq r \leq k \Big\rangle .
\end{multline*}
In particular,
\[
\prod_{i=1}^{k} x_i^{j-1} \cdot \prod_{(\alpha_{1},\dots, \alpha_{k})\in\F_2^k{\setminus}\Gamma} (\alpha_{1}x_1 + \cdots + \alpha_{k}x_k)^{j} \in 
\Big\langle \sum_{s = 0}^{\ell}  x_{r}^{s} \otimes w_{\ell-s} 	\: : \:	1 \leq r \leq k \Big\rangle.
\]
This is a contradiction with the assumption \eqref{eq : assumption in main theorem}, and we can conclude that the $4$-tuple $(d,\ell,k,j)$ is mass assignment admissible. \qed

\subsection{Proof of Corollary \ref{cor : ham-sandwich}}
\label{Proof of Corollary ham-sandwich}

In the case when $\ell=d$ the admissibility of $(d,d,1,d)$ is just the classical ham sandwich theorem.
Thus, the proof which we present is new for  $1\leq \ell \leq d-1$.

\medskip
Since we are in the situation where $k=1$ and $j=d$, then $e_{1,d}=x_1^{d-1}$, and the ideal $\II_{d,\ell,1}$ is the principal ideal generated by the polynomial
$
p:=\sum_{s = 0}^{\ell}  x_{1}^{s} \, w_{\ell-s}
$.
According to the Theorem \ref{theorem : Main} the $4$-tuple $( d, \ell , 1 , d )$ is mass assignment admissible if $e_{1,d}\notin \II_{d,\ell,1}$, or equivalently $p\nmid e_{1,d}$.
This means that 
\begin{equation}
	\label{non-divisibility}
	\sum_{s = 0}^{\ell}  x_{1}^{s} \, w_{\ell-s}\nmid x_1^{d-1}
\end{equation}
in the ring $R_{d,\ell,1}:=\F_2[x_1, w_1, \dots ,w_{\ell}, \overline{w}_1, \dots , \overline{w}_{d-\ell}]/I_{d,\ell}$.
Therefore, by verifying the claim of the relation \eqref{non-divisibility} we will complete the proof of the corollary.

\medskip
\noindent{\em Proof of \eqref{non-divisibility}}: 
From \cite[Proof of Prop.\,4.1 with  $\eta=\gamma_{\ell}^d$ and $\xi$ trivial]{Crabb2013} we have that 
\[
x_1^{d-1} =\Big(\sum_{s = 0}^{\ell}  x_{1}^{s} \, w_{\ell-s} \Big)\, q + \big(a_{\ell-1}x_1^{\ell-1}+ \cdots + a_1x_1 +a_0\big)
\]
where the coefficients $a_{\ell-1},\cdots,a_0$ of the remainder  can be expressed by:
\[
a_r = \overline{w}_{d-r-1}+w_1\overline{w}_{d-r-2}+ \dots + w_{\ell-r-1}\overline{w}_{d-\ell},
\]
for $0\leq r\leq \ell-1$.
Since $a_{\ell-1}=\overline{w}_{d-\ell}\neq 0$ the remainder does not vanish and therefore the relation \eqref{non-divisibility} holds.
This argument will be used once again in the proof of Lemma \ref{lem : intersection of ideal and subring} as relation \eqref{division of x^(d-1)}. \qed
 
\subsection{Proof of Theorem \ref{theorem : Main numerical}}
\label{Proof of Main numerical theorem}

Thus, let us assume that $d\geq 2$, $k\geq 1$, $j\geq 1$, $\ell\geq 1$ and $t\geq 0$, $r\geq 0$ are integers with $1\leq k\leq \ell\leq d$.
Set $j=2^t+r$ with $0\leq r\leq 2^{t}-1$, and in addition  assume, without loss of generality, that $d= 2^{t+k-1}+r$.

\medskip
In order to prove that the $4$-tuple $(d,\ell,k,j)$ is mass assignment admissible we show that 
\[
e_{k,j}\notin \II_{d,\ell,k}\subseteq R_{d,\ell,k}\cong (\F_2[w_1, \dots ,w_{\ell}, \overline{w}_1, \dots , \overline{w}_{d-\ell}]/I_{d,\ell})[x_1,\dots,x_k],
\]
which, according to Theorem \ref{theorem : Main}, is sufficient.

\medskip
Since the element $e_{k,j}$ is actually a polynomial with coefficients only in $\F_2$, that is $e_{k,j}\in \F_2[x_1,\dots,x_k]\subseteq R_{d,\ell,k}$, we first analyse the intersection $\F_2[x_1,\dots,x_k]\cap  \II_{d,\ell,k}$ of the subring $\F_2[x_1,\dots,x_k]$ of the ring $R_{d,\ell,k}$ and the ideal $\II_{d,\ell,k}$.
For simplicity, as in the previous section, we denote the generators of the ideal $\II_{d,\ell,k}$ by
\[
\beta_i:=\sum_{s = 0}^{\ell}  x_{i}^{s} \, w_{\ell-s} ,
\]
where $1 \leq i \leq k$. 
Hence,  $\II_{d,\ell,k}=\langle \beta_1,\dots,\beta_k \rangle$.

\begin{lemma}
\label{lem : intersection of ideal and subring}
$e_{k,j}\notin \langle \beta_1,\dots,\beta_k \rangle\subseteq R_{d,\ell,k}$ if and only if $e_{k,j}\notin\langle x_1^d,\dots,x_k^d\rangle\subseteq\F_2[x_1,\dots,x_k]$.	
\end{lemma}
\begin{proof}
First, we observe that $\langle x_1^d,\dots,x_k^d \rangle\subseteq \langle \beta_1,\dots,\beta_k \rangle$. Indeed, for every $1\leq i\leq k$ we have that
\[
x_i^d= \Big(\sum_{r = 0}^{d-\ell}  x_{i}^{r} \,  \overline{w}_{d-\ell-r}\Big)\cdot\Big(\sum_{s = 0}^{\ell}  x_{i}^{s} \, w_{\ell-s} \Big)
= \Big(\sum_{r = 0}^{d-\ell}  x_{i}^{r} \,  \overline{w}_{d-\ell-r}\Big)\,\beta_i\in  \langle \beta_1,\dots,\beta_k \rangle.
\]
Consequently, by contraposition, we get that
\[
e_{k,j}\notin \langle \beta_1,\dots,\beta_k \rangle \ \Longrightarrow \ e_{k,j}\notin\langle x_1^d,\dots,x_k^d\rangle. 
\]

\medskip
To prove the opposite implication, it suffices to show that the inclusion
\[
\F_2[x_1,\dots,x_k] \longrightarrow
(\F_2[w_1, \dots ,w_{\ell}, \overline{w}_1, \dots , \overline{w}_{d-\ell}]/I_{d,\ell})[x_1,\dots,x_k]
\]
induces a monomorphism
\begin{multline}
\label{eq -- injection}
\F_2[x_1,\dots,x_k]/ \langle x_1^d,\dots,x_k^d\rangle 
\longrightarrow \\
(\F_2[w_1, \dots ,w_{\ell}, \overline{w}_1, \dots , \overline{w}_{d-\ell}]/I_{d,\ell})[x_1,\dots,x_k]/\langle \beta_1,\dots,\beta_k \rangle	.
\end{multline}

\medskip
First, let us observe that in the ring $R_{d,\ell,k}$ the following identity holds
\begin{equation}
\label{division of x^(d-1)}	
x_i^{d-1} =\Big(\sum_{s = 0}^{\ell}  x_{i}^{s} \, w_{\ell-s} \Big)\, q + \big(a_{\ell-1}x_i^{\ell-1}+ \cdots + a_1x_i +a_0\big)
\end{equation} 
where $q\in R_{d,\ell,k}$ is a polynomial of degree $d-1-\ell$, and $1\leq i\leq k$.
The coefficients $a_{\ell-1},\cdots,a_0$ of the remainder in the previous equation can be explicitly computed, as shown in \cite[Proof of Prop.\,4.1 with  $\eta=\gamma_{\ell}^d$ and $\xi$ trivial]{Crabb2013}.
In particular,  for $0\leq r\leq \ell-1$:
\[
a_r = \overline{w}_{d-r-1}+w_1\overline{w}_{d-r-2}+ \dots + w_{\ell-r-1}\overline{w}_{d-\ell}.
\]
Since $\overline{w}_{s}=0$ for every $s\geq d-\ell+1$ we have that $a_r= w_{\ell-r-1}\overline{w}_{d-\ell}$ and specifically $a_{\ell-1}= \overline{w}_{d-\ell}\neq 0$.
Thus,
\begin{multline*}
	x_i^{d-1} + \langle \beta_1,\dots,\beta_k \rangle   
= \\
\overline{w}_{d-\ell} ( x_i^{\ell -1} + w_1x_i^{\ell-2} + \cdots + w_{\ell-2}x_i^{1} + w_{\ell -1} )
+ \langle \beta_1,\dots,\beta_k \rangle  
\neq  
\langle \beta_1,\dots,\beta_k \rangle  .
\end{multline*}

\medskip
Now, we show the injectivity of the map \eqref{eq -- injection}.
Denote by $I$ the kernel ideal of the map 
\begin{multline*}
\F_2[x_1,\dots,x_k] \longrightarrow \F_2[x_1,\dots,x_k]/ \langle x_1^d,\dots,x_k^d\rangle 
\longrightarrow \\
(\F_2[w_10, \dots ,w_{\ell}, \overline{w}_1, \dots , \overline{w}_{d-\ell}]/I_{d,\ell})[x_1,\dots,x_k]/\langle \beta_1,\dots,\beta_k \rangle.
\end{multline*}
In particular, $\langle x_1^d,\dots,x_k^d\rangle\subseteq I$.
Furthermore assume that 
\[
0\neq p=\sum_{(r_1,\dots,r_k)\in A\subseteq \{0,\dots, d-1\}^k} a_{r_1,\dots,r_k}\,x_1^{r_1}\cdots x_k^{r_k}\in I,
\]
for some index set $A\subseteq \{0,\dots, d-1\}^k$.
If $z=\min\{z : a_{r_1,\dots,r_{k-1},z}\neq 0\}$ then the polynomial $p\cdot x_k^{d-1-z}$ has all monomials with the exponent in $x_k$ at least $d-1$. 
Continuing in this way along the variables $x_{k-1}$, $x_{k-2}$, all the way down to $x_1$, we get that $x_1^{d-1}\cdots x_k^{d-1}\in I$.
In particular, this means that
\begin{multline*}
\langle \beta_1,\dots,\beta_k \rangle  =	x_1^{d-1}\cdots x_k^{d-1}+\langle \beta_1,\dots,\beta_k \rangle =\\
\overline{w}_{d-\ell}^k  \prod_{1 \leq i \leq k} \big(  x_i^{\ell -1} + w_1x_i^{\ell-2} + \cdots + w_{\ell -1} \big) 
+ \langle \beta_1,\dots,\beta_k \rangle  \\
\neq  \langle \beta_1,\dots,\beta_k \rangle,  
\end{multline*}
because $1\leq k\leq \ell$ and $\overline{w}_{d-\ell}^{\ell}\neq 0$ (see for example \cite{Hiller-1980}).
We have reached a contradiction with the assumption of the existence of a polynomial $p$ in the ideal $I$. 
Hence, the injectivity of the map \eqref{eq -- injection} is confirmed and the proof of the lemma is completed.
\end{proof}

\medskip
Using Lemma \ref{lem : intersection of ideal and subring} we complete the proof of Theorem \ref{theorem : Main numerical} by proving the following fact.

\begin{lemma}
\label{lem : Euler class not in the ideal}
$e_{k,j}\notin \II_{d,\ell,k}$.	
\end{lemma}
\begin{proof}
According to Lemma \ref{lem : intersection of ideal and subring} it suffices to show that 
$
e_{k,j}\notin\langle x_1^d,\dots,x_k^d\rangle\subseteq\F_2[x_1,\dots,x_k]
$.

\medskip
First, we transform the polynomial $e_{k,j}$ in as follows:
\begin{align*}
e_{k,j} &=	
\frac{1}{x_1\cdots x_k} \prod_{(\alpha_{1},\dots, \alpha_{k})\in\F_2^k{\setminus}\{ (0,\ldots,0) \}} (\alpha_{1}x_1 + \cdots + \alpha_{k}x_k)^{j}\\
&=\frac{1}{x_1\cdots x_k}\Big( \prod_{(\alpha_{1},\dots, \alpha_{k})\in\F_2^k{\setminus}\{ (0,\ldots,0) \}} (\alpha_{1}x_1 + \cdots + \alpha_{k}x_k)\Big)^{j}\\
&=\frac{1}{x_1\cdots x_k}\cdot \Delta_k^j,
\end{align*}
where
\[
\Delta_k:= \prod_{(\alpha_{1},\dots, \alpha_{k})\in\F_2^k{\setminus}\{ (0,\ldots,0) \}} (\alpha_{1}x_1 + \cdots + \alpha_{k}x_k)
\]
is the Dickson polynomial of maximal degree.
Furthermore, the Dickson polynomial $\Delta_k$ can be presented as a polynomial in $x_k$ by:
\begin{align}\label{eq : delta}
\Delta_k &=	\prod_{(\alpha_{1},\dots, \alpha_{k})\in\F_2^k{\setminus}\{ (0,\ldots,0) \}} (\alpha_{1}x_1 + \cdots + \alpha_{k}x_k) \nonumber \\
&= \Delta_{k-1}x_k\prod_{(\alpha_{1},\dots, \alpha_{k})\in\F_2^{k-1}{\setminus}\{ (0,\ldots,0) \}} (\alpha_{1}x_1 + \cdots + \alpha_{k-1}x_{k-1}+x_k)\nonumber \\
&=\Delta_{k-1}x_k\Big(\sum_{i=0}^{k-1}D_{k-1,i}x_k^{2^i-1} \Big)
\end{align}
where $D_{k-1,0},\dots, D_{k-1,k-2}$ are Dickson polynomials in variables $x_1,\dots,x_{k-1}$.
In particular, $D_{k-1,0}=\Delta_{k-1}$ and $D_{k-1,k-1}=1$.
For more details on Dickson polynomials consult for example \cite{Wilkerson1983}.

\medskip
Now, we start the proof of the claim $e_{k,j}\notin\langle x_1^d,\dots,x_k^d\rangle$ using induction on $k$.
In the case $k=1$ we have that $\Delta_1=x_1$ and $d=2^t+r=j$.
Then our claim reduces to the obvious fact that $e_{i,j}=x_1^{j-1}\notin \langle x_1^j\rangle$.
Let us assume, as an induction hypothesis, that
\begin{equation}
	\label{eq : IHH}
	e_{k-1,j}\notin \langle x_1^{2^{t+k-2}+r},\dots,x_{k-1}^{2^{t+k-2}+r}\rangle.
\end{equation}

\medskip
For the induction step we represent the integer $j$ in its binary form as:
\[
j=2^{t_1}+2^{t_2}+\cdots +2^{t_a},
\]
where $a\geq 1$ and $t=t_1>t_2>\cdots > t_a\geq 0$.
In particular, $r=2^{t_2}+\cdots +2^{t_a}$.
Now from \eqref{eq : delta} it follows that
\begin{equation}\label{eq : delta-j}
\Delta_k^j = \Delta_{k-1}^jx_k^j\Big(\sum_{i=0}^{k-1}D_{k-1,i}\,x_k^{2^i-1}	\Big)^j= \Delta_{k-1}^jx_k^j \prod_{b=1}^a \Big(\sum_{i=0}^{k-1}D_{k-1,i}^{2^{t_b}}\,x_k^{(2^i-1)2^{t_b}}	\Big).
\end{equation}
A typical monomial in the expansion of the left hand side of \eqref{eq : delta-j} is of the form
\begin{align*}
m &=\Delta_{k-1}^jx_k^j \cdot D_{k-1,i_1}^{2^{t_1}}\,x_k^{(2^{i_1}-1)2^{t_1}}\cdot D_{k-1,i_2}^{2^{t_2}}\,x_k^{(2^{i_2}-1)2^{t_2}}\cdots D_{k-1,i_a}^{2^{t_a}}\,x_k^{(2^{i_a}-1)2^{t_a}}\\
&=\Delta_{k-1}^jD_{k-1,i_1}^{2^{t_1}}D_{k-1,i_2}^{2^{t_2}}\cdots D_{k-1,i_a}^{2^{t_a}}x_k^E,
\end{align*}
where $0\leq i_1,\dots, i_a\leq k-1$, and 
\[
E=j+(2^{i_1+t_1}-2^{t_1})+(2^{i_2+t_2}-2^{t_2})+\cdots+(2^{i_a+t_a}-2^{t_a})=2^{i_1+t_1}+\cdots+2^{i_a+t_a}.
\]
Observe that,
\[
E = 2^{i_1+t_1}+2^{i_2+t_2}+ \cdots + 2^{i_a+t_a} 
   = 2^{t_1+k-1} + 2^{t_2} + \cdots + 2^{t_a}
\]
if and only if
\[
2^{t_2}(2^{i_2}-1)+\cdots+2^{t_a}(2^{i_a}-1)=2^{t_1+i_1}(2^{k-1-i_1}-1)
\]
if and only if
\[
i_1=k-1 \quad\text{and}\quad i_2=\dots=i_a=0,
\]
because $t_1>t_2>\cdots > t_a\geq 0$. 
Thus, in the expansion of the polynomial $\Delta_k^j$ the (one and only) monomial of degree $2^{t_1+k-1}+r=2^{t+k-1}+r$ in the variable $x_k$ is of the form $\Delta_{k-1}^{j+r}\,x_k^{2^{t_1+k-1}+r}$.
In other words, in the expansion of the polynomial $e_{k,j}$ the (one and only) monomial of degree $2^{t_1+k-1}+r-1=2^{t+k-1}+r-1$ in the variable $x_k$ is of the form $\frac{1}{x_1\cdots x_{k-1}}\cdot \Delta_{k-1}^{j+r}\,x_k^{2^{t+k-1}+r-1}$.

\medskip\noindent
In the final step of the proof we use the induction hypothesis \eqref{eq : IHH} with $j'=2^{t+1}+r$ where $0\leq r\leq 2^{t}-1$, that is
\[
e_{k-1,j'}=\frac{1}{x_1\cdots x_{k-1}}\cdot \Delta_{k-1}^{j'}=\frac{1}{x_1\cdots x_{k-1}}\cdot \Delta_{k-1}^{2^{t+1}+r}\notin \langle x_1^{2^{t+k-1}+r},\dots,x_{k-1}^{2^{t+k-1}+r}\rangle.
\]
Since $j+r=2^t+2r\leq j'=2^{t+1}+r$ we conclude that 
\[
\frac{1}{x_1\cdots x_{k-1}}\cdot \Delta_{k-1}^{j+r}\notin \langle x_1^{2^{t+k-1}+r},\dots,x_{k-1}^{2^{t+k-1}+r}\rangle.
\]
 Hence, $e_{k,j}=\frac{1}{x_1\cdots x_k}\cdot \Delta_k^j$ has a (non-zero) monomial $x^{\alpha_1}_1\cdots x^{\alpha_{k-1}}_{k-1}x_k^{2^{t+k-1}+r-1}$ where $\alpha_i\leq  2^{t+k-1}+r-1$ for all $1\leq i\leq k-1$.
Therefore, $e_{k,j}\notin\langle x_1^d,\dots,x_k^d\rangle$, and the induction step is completed.

\medskip
This proof adds the missing argument in the proof of \cite[Thm.\,39]{ManiLevitska-Vrecica-Zivaljevic-2006} and corrects the final steps of the proof of \cite[Thm.\,3.2]{Blagojevic-Frick-Haase-Ziegler-2018}. 
\end{proof}

\medskip
The proof of Theorem \ref{theorem : Main numerical} is now finished.

\subsection*{Acknowledgements.}
We are grateful to Omar Antolin for a useful suggestion regarding the definition of the notion of mass assignment and to Pablo Sober\'on for pointing out that the lower bound on dimension from the classical case extends to mass assignment context. 
The authors would also like to thank Matija Blagojevi\'c for the considerable improvements to the manuscript.

\end{document}